\documentclass[reqno]{amsart}
\usepackage{amsmath, amsthm, amssymb, amstext}

\usepackage{hyperref,xcolor}
\hypersetup{pdfborder={0 0 0},colorlinks}
\usepackage{enumitem}
\allowdisplaybreaks
\usepackage{todonotes}

\newtheorem{theorem}{Theorem}
\newtheorem{remark}[theorem]{Remark}
\newtheorem{lemma}[theorem]{Lemma}
\newtheorem{proposition}[theorem]{Proposition}

\newcommand{\N}{\mathbb{N}}
\newcommand{\R}{\mathbb{R}}
\renewcommand{\l}{\left}
\renewcommand{\r}{\right}

\numberwithin{theorem}{section}
\numberwithin{equation}{section}

\title[Degenerate Kirchhoff problems]{Degenerate Kirchhoff problems with nonlinear Neumann boundary condition}

\author[F. Borer]{Franziska Borer}
\address[F. Borer]{Technische Universit\"{a}t Berlin, Institut f\"{u}r Mathematik, Stra\ss e des 17. Juni 136, 10623 Berlin, Germany}
\email{borer@math.tu-berlin.de}

\author[M.T.O. Pimenta]{Marcos T.O. Pimenta}
\address[M.T.O. Pimenta]{Departamento de Matem\'atica e Computa\c{c}\~ao, Universidade Estadual Paulista - Unesp, CEP: 19060-900, Presidente Prudente - SP, Brazil}
\email{marcos.pimenta@unesp.br}

\author[P. Winkert]{Patrick Winkert}
\address[P. Winkert]{Technische Universit\"{a}t Berlin, Institut f\"{u}r Mathematik, Stra\ss e des 17.\,Juni 136, 10623 Berlin, Germany}
\email{winkert@math.tu-berlin.de}



\begin{document}
\begin{abstract}
	In this paper we consider degenerate Kirchhoff-type equations of the form
	\begin{equation*}
		\begin{aligned}
			-\phi(\Xi(u)) \l(\mathcal{A}(u)-|u|^{p-2}u\r) & = f(x,u)\quad &  & \text{in } \Omega,\\
			\phi (\Xi(u)) \mathcal{B}(u) \cdot \nu &= g(x,u)          &  & \text{on } \partial\Omega,
		\end{aligned}
	\end{equation*}
	where $\Omega\subseteq \R^N$, $N\geq 2$, is a bounded domain with Lipschitz boundary $\partial\Omega$, $\mathcal{A}$ denotes the double phase operator given by
	\begin{align*}
		\mathcal{A}(u)=\operatorname{div} \left(|\nabla u|^{p-2}\nabla u + \mu(x) |\nabla u|^{q-2}\nabla u \right)\quad \text{for }u\in W^{1,\mathcal{H}}(\Omega),
	\end{align*}
	$\nu(x)$ is the outer unit normal of $\Omega$ at $x \in \partial\Omega$,
	\begin{equation*}
		\begin{aligned}
			\mathcal{B}(u)&=|\nabla u|^{p-2}\nabla u + \mu(x) |\nabla u|^{q-2}\nabla u,\\
			\Xi(u)&= \int_\Omega \left(\frac{|\nabla u|^p+|u|^p}{p}+\mu(x) \frac{|\nabla u|^q}{q}\right)\,\mathrm{d} x,
		\end{aligned}
	\end{equation*}
	 $1<p<N$, $p<q<p^*=\frac{Np}{N-p}$, $0 \leq \mu(\cdot)\in L^\infty(\Omega)$, $\phi(s) = a + b s^{\zeta-1}$ for $s\in\R$ with $a \geq 0$, $b>0$ and $\zeta \geq 1$, and $f\colon\Omega\times\R\to\R$, $g\colon\partial\Omega\times\R\to\R$ are Carath\'{e}odory functions that grow superlinearly and subcritically. We prove the existence of a nodal ground state solution to the problem above, based on variational methods and minimization of the associated energy functional $\mathcal{E}\colon W^{1,\mathcal{H}}(\Omega) \to\R$ over the constraint set
	\begin{align*}
		\mathcal{C}=\Big\{u \in  W^{1,\mathcal{H}}(\Omega)\colon u^{\pm}\neq 0,\, \l\langle \mathcal{E}'(u),u^+ \r\rangle= \l\langle \mathcal{E}'(u),-u^- \r\rangle=0 \Big\},
	\end{align*}
	whereby $\mathcal{C}$ differs from the well-known nodal Nehari manifold due to the nonlocal character of the problem.
\end{abstract}

\maketitle

%***********************************************************************************
\section{Introduction and main results}
%***********************************************************************************

Let $\Omega\subseteq \R^N$, $N\geq 2$, be a bounded domain with a Lipschitz boundary $\partial\Omega$. In this work we study the following degenerate Kirchhoff problem with a nonlinear Neumann boundary condition of the shape

\begin{equation}
	\label{problem}
	\begin{aligned}
		-\phi(\Xi(u)) \l(\mathcal{A}(u)-|u|^{p-2}u\r) & = f(x,u)\quad &  & \text{in } \Omega,\\
		\phi (\Xi(u)) \mathcal{B}(u) \cdot \nu &= g(x,u)          &  & \text{on } \partial\Omega,
	\end{aligned}
\end{equation}
where
\begin{align}\label{operator_double_phase}
	\mathcal{A}(u)=\operatorname{div} \left(|\nabla u|^{p-2}\nabla u + \mu(x) |\nabla u|^{q-2}\nabla u \right)\quad \text{for }u\in W^{1,\mathcal{H}}(\Omega)
\end{align}
is the double phase operator, $W^{1,\mathcal{H}}(\Omega)$ denotes the Musielak-Orlicz Sobolev space (for a precise definition see Section \ref{Section_2}),
$\nu(x)$ is the outer unit normal of $\Omega$ at $x \in \partial\Omega$,
\begin{equation}\label{new-16}
	\begin{aligned}
		\mathcal{B}(u)&=|\nabla u|^{p-2}\nabla u + \mu(x) |\nabla u|^{q-2}\nabla u,\\
		\Xi(u)&= \int_\Omega \left(\frac{|\nabla u|^p+|u|^p}{p}+\mu(x) \frac{|\nabla u|^q}{q}\right)\,\mathrm{d} x,
	\end{aligned}
\end{equation}
$1<p<N$, $p<q<p^*=\frac{Np}{N-p}$, $0 \leq \mu(\cdot)\in L^\infty(\Omega)$, $\phi(s) = a + b s^{\zeta-1}$ for $s\in\R$ with $a \geq 0$, $b>0$ and $\zeta \geq 1$, and $f\colon\Omega\times\R\to\R$, $g\colon\partial\Omega\times\R\to\R$ are Carath\'{e}odory functions that grow superlinearly and subcritically. We refer to hypotheses \eqref{H3} for the exact conditions on $f$ and $g$.

We emphasize that problem \eqref{problem} combines several interesting phenomena. First, the occurring operator is the double phase operator given in \eqref{operator_double_phase} which is closely related to the energy functional
\begin{align}\label{integral_minimizer}
	J(u)= \int_\Omega \big(|\nabla  u|^p+\mu(x)|\nabla  u|^q\big)\,\mathrm{d} x.
\end{align}
Functionals of the form \eqref{integral_minimizer} were first mentioned by Zhikov \cite{Zhikov-1986} in 1986 in order to characterize models for strongly anisotropic materials in the context of homogenization and elasticity. It also occurs in the study of duality theory and of the Lavrentiev gap phenomenon, see the works of Zhikov \cite{Zhikov-1995,Zhikov-2011}. From the mathematical point of view, first regularity properties of local minimizers of functionals like \eqref{integral_minimizer} have been proved in the papers by Baroni-Colombo-Mingione \cite{Baroni-Colombo-Mingione-2015,Baroni-Colombo-Mingione-2018} and Colombo-Mingione \cite{Colombo-Mingione-2015a,Colombo-Mingione-2015b}. We also mention the pioneering works of Marcellini \cite{Marcellini-1991,Marcellini-1989} for integral functionals with nonstandard growth condition.

A second interesting phenomenon is the appearance of the nonlocal Kirchhoff term which is generated by the function $\phi(s) = a + b s^{\zeta-1}$ for $s\in\R$ with $a \geq 0$, $b>0$ and $\zeta \geq 1$. Problems of this type go back to a model which was first presented by Kirchhoff \cite{Kirchhoff-1883} in 1883. He proposed the model problem
\begin{align*}
	\rho \frac{\partial^2 u}{\partial t^2}-\left(\frac{\rho_0}{h}+\frac{E}{2L}\int_0^L \left|\frac{\partial u}{\partial x}\right|^2 \,\mathrm{d} x\right)\frac{\partial^2 u}{\partial x^2}=0,
\end{align*}
which is a generalization of the D'Alembert equation. It should be pointed out that problem \eqref{problem} generalizes several models which describe interesting phenomena studied on mathematical physics. Since the constant $a$ in the Kirchhoff function could be zero, problem \eqref{problem} is called degenerate which creates the most interesting models in the applications. Indeed, if $p=2$ and $\mu(x)\equiv 0$, the transverse oscillations of a stretched string with nonlocal flexural rigidity depends continuously on the Sobolev deflection norm of $u$ via $\phi(\int_\Omega |\nabla u|^2\,\mathrm{d} x)$ meaning that $\phi(0)=0$ is nothing less than the base tension of the string is zero. Moreover, there is a large list of references dealing with different type of Kirchhoff problems. We just refer to some of the most famous ones, that is, Alves-Corr\^{e}a-Ma \cite{Alves-Correa-Ma-2005}, Autuori-Pucci-Salvatori \cite{Autuori-Pucci-Salvatori-2010}, D'Ancona-Spagnolo \cite{DAncona-Spagnolo-1992}, Figueiredo \cite{Figueiredo-2013},  Fiscella-Valdinoci \cite{Fiscella-Valdinoci-2014}, He-Zou \cite{He-Zou-2012}, Lions \cite{Lions-1978}, Mao-Zhang \cite{Mao-Zhang-2009}, Mingqi-R\u{a}dulescu-Zhang \cite{Mingqi-Radulescu-Zhang-2019}, Perera-Zhang \cite{Perera-Zhang-2006}, Pucci-Xiang-Zhang \cite{Pucci-Xiang-Zhang-2015}, Xiang-Zhang-R\u{a}dulescu \cite{Xiang-Zhang-Radulescu-2016}, see also the references therein.

Even though the literature for Kirchhoff problems is very large, it is still very manageable for Kirchhoff double phase settings. The first work in this field was written by Fiscella-Pinamonti \cite{Fiscella-Pinamonti-2023} who considered Kirchhoff-type problems of the style
\begin{align*}%\label{problem2534}
	-m \left[\int_\Omega \left( \frac{|\nabla u|^p}{p} + \mu(x) \frac{|\nabla u|^q}{q}\right)\,\mathrm{d} x\right]\mathcal{A}(u) =f(x,u) \quad \text{in } \Omega,\quad u  = 0 \quad \text{on } \partial\Omega,
\end{align*}
whereby they suppose the Ambrosetti-Rabinowitz condition along with a subcritical growth on the perturbation $f\colon \Omega\times\R\to\R$. Their existence result is mainly based on a mountain pass geometry of the problem. Singular Kirchhoff-type problems involving the double phase operator have been studied by Arora-Fiscella-Mukherjee-Win\-kert \cite{ Arora-Fiscella-Mukherjee-Winkert-2023} (see also \cite{Arora-Fiscella-Mukherjee-Winkert-2022} by the same authors in the critical case) in order to get positive solutions of the problem
\begin{align*}%\label{kirchhoff-critical}
	-m \left[\int_\Omega \left( \frac{|\nabla u|^p}{p} + \mu(x) \frac{|\nabla u|^q}{q}\right)\,\mathrm{d} x\right]\mathcal{A}(u) = \lambda u^{-\gamma} +u^{r-1} \quad \text{in } \Omega,
	\quad u\mid_{\partial\Omega} = 0
\end{align*}
by using the Nehari manifold and minimization arguments. Other results can be found in the papers of Cen-Vetro-Zeng \cite{Cen-Vetro-Zeng-2023}, Cheng-Bai \cite{Cheng-Bai-2024}, Ho-Winkert \cite{Ho-Winkert-2023}, see also Isernia-Repov\v{s} \cite{Isernia-Repovs-2021} for $(p,q)$-problems in the whole space.

Finally, a third interesting phenomenon in problem \eqref{problem} is the appearance of a nonlinear Neumann boundary condition. Such a boundary condition makes the treatment of problem \eqref{problem} much more complicated. As far as we know there is just one paper in the direction of Kirchhoff double phase problems with a nonlinear boundary condition published by Fiscella-Marino-Pinamonti-Verzellesi \cite{Fiscella-Marino-Pinamonti-Verzellesi-2024} studying the problem
\begin{align*}%\label{kirchhoff-critical}
	-M \left[\int_\Omega \left( \frac{|\nabla u|^p}{p} + \mu(x) \frac{|\nabla u|^q}{q}\right)\,\mathrm{d} x\right]\mathcal{A}(u) &= h_1(x,u) \quad \text{in } \Omega,\\
	-M \left[\int_\Omega \left( \frac{|\nabla u|^p}{p} + \mu(x) \frac{|\nabla u|^q}{q}\right)\,\mathrm{d} x\right]\mathcal{B}(u)
	&=h_2(x,u)\quad \text{in }\partial\Omega
\end{align*}
with $\mathcal{B}$ as in \eqref{new-16}. The authors prove several existence results for different structures of $h_1$ and $h_2$ based on variational tools and a version of the fountain theorem. Further works for double phase problems with nonlinear Neumann boundary condition but without Kirchhoff term have been published by Guarnotta-Livrea-Winkert \cite{Guarnotta-Livrea-Winkert-2023}, Papageorgiou-R\u{a}dulescu-Repov\v{s} \cite{Papageorgiou-Radulescu-Repovs-2020}, Papageorgiou-Vetro-Vetro \cite{Papageorgiou-Vetro-Vetro-2021} and Zeng-R\u{a}dulescu-Winkert \cite{Zeng-Radulescu-Winkert-2022,Zeng-Radulescu-Winkert-2024}, see also the papers of Corr\^{e}a-Nascimento \cite{Correa-Nascimento-2009} and Dai-Ma \cite{Dai-Ma-2011} for Kirchhoff problems with Neumann boundary condition, but without a double phase operator.

Now we are going to formulate our assumptions on the data of problem \eqref{problem} and present our main result in this paper. We assume the subsequent conditions:

\begin{enumerate}[label=\textnormal{(A$_1$)},ref=\textnormal{A$_1$}]
	\item\label{H1}
		$1<p<N$, $p<q<p^*=\frac{Np}{N-p}$ and $0 \leq \mu(\cdot)\in L^\infty(\Omega)$.
\end{enumerate}

\begin{enumerate}[label=\textnormal{(A$_2$)},ref=\textnormal{A$_2$}]
	\item\label{H2}
		$\phi\colon  [0,\infty) \to [0,\infty)$ is a continuous function given by
		\begin{align}\label{new-17}
			\phi(s) = a + b s^{\zeta-1}
		\end{align}
		with $a \geq 0$, $b>0$ and $\zeta \geq 1$ is such that $q\zeta<p_*=\frac{(N-1)p}{N-p}$.
\end{enumerate}

\begin{enumerate}[label=\textnormal{(A$_3$)},ref=\textnormal{A$_3$}]
	\item\label{H3}
		$f\colon \Omega \times \R \to \R$ and $g\colon\partial\Omega\times\R\to\R$ are Carath\'{e}odory functions such that the following hold:
		\begin{enumerate}[label=\textnormal{(\roman*)},ref=\textnormal{\roman*}]
			\item\label{H3i}
				there exist constants $c_1,c_2>0$ such that
				\begin{align*}
					|f(x,s)| & \leq c_1\left( 1 +|s|^{r_1-1}\right) \quad \text{for a.a.\,}x\in\Omega,\\
					|g(x,s)| & \leq c_2\left( 1 +|s|^{r_2-1}\right) \quad \text{for a.a.\,}x\in\partial \Omega,
				\end{align*}
				for all $s\in \R$, where $r_1<p^*$ and $r_2<p_*$;
			\item\label{H3ii}
				\begin{align*}
					& \lim_{s \to \pm \infty}\,\frac{f(x,s)}{|s|^{q\zeta-2}s}=+\infty \quad\text{uniformly for a.a.\,}x\in\Omega,\\
					& \lim_{s \to \pm \infty}\,\frac{g(x,s)}{|s|^{q\zeta-2}s}=+\infty \quad\text{uniformly for a.a.\,}x\in\partial \Omega;
				\end{align*}
			\item\label{H3iii}
				\begin{align*}
					& \lim_{s \to 0}\,\frac{f(x,s)}{|s|^{p\zeta-2}s}=0 \quad\text{uniformly for a.a.\,}x\in\Omega,\\
					& \lim_{s \to 0}\,\frac{g(x,s)}{|s|^{p\zeta-2}s}=0 \quad\text{uniformly for a.a.\,}x\in\partial\Omega;
				\end{align*}
			\item\label{H3iv}
				the functions
				\begin{align*}
					s\mapsto f(x,s)s- q\zeta F(x,s)
					\quad\text{and}\quad
					s\mapsto g(x,s)s- q\zeta G(x,s)
				\end{align*}
				are nondecreasing on $[0,\infty)$ and nonincreasing on $(-\infty,0]$ for a.a.\,$x\in\Omega$ and for a.a.\,$x\in\partial\Omega$, respectively, where
				\begin{align*}
					F(x,s)=\int_0^s f(x,t)\,\mathrm{d} t
					\quad\text{and}\quad
					G(x,s)=\int_0^s g(x,t)\,\mathrm{d} t;
				\end{align*}
			\item\label{H3v}
				the functions
				\begin{align*}
					s\mapsto \frac{f(x,s)}{|s|^{q\zeta-1}}
					\quad\text{and}\quad
					s\mapsto \frac{g(x,s)}{|s|^{q\zeta-1}}
				\end{align*}
				are strictly increasing on $(-\infty,0)$ and on $(0,+\infty)$ for a.a.\,$x\in\Omega$ and for a.a.\,$x\in\partial\Omega$, respectively, where $\zeta$ is from \eqref{H2}.
		\end{enumerate}
\end{enumerate}

\begin{remark}\label{remark-H2}
	It should be mentioned that from \eqref{H3}\eqref{H3i} and \eqref{H3ii} we can conclude that $q\zeta < \min\{r_1,r_2\}$. Moreover, due to \eqref{H3}\eqref{H3i} and \eqref{H3ii}, it holds that
	\begin{equation}\label{superlinear-F}
		\begin{aligned}
			&\lim_{s \to \pm \infty}\,\frac{F(x,s)}{|s|^{q\zeta}}=+\infty \quad\text{uniformly for a.a.\,}x\in\Omega,\\
			&\lim_{s \to \pm \infty}\,\frac{G(x,s)}{|s|^{q\zeta}}=+\infty \quad\text{uniformly for a.a.\,}x\in\partial\Omega.
		\end{aligned}
	\end{equation}
\end{remark}

\begin{remark}
	If $a>0$, one can suppose weaker assumptions as in \eqref{H3}\eqref{H3iii} of the form
	\begin{align*}
		& \lim_{s \to 0}\,\frac{f(x,s)}{|s|^{p-2}s}=0 \quad\text{uniformly for a.a.\,}x\in\Omega,\\
		& \lim_{s \to 0}\,\frac{g(x,s)}{|s|^{p-2}s}=0 \quad\text{uniformly for a.a.\,}x\in\partial\Omega.
	\end{align*}
	At some places the proofs become easier under these assumptions. We will not consider this case further.
\end{remark}

A weak solution of problem \eqref{problem} is to be understood as a function $u \in  W^{1,\mathcal{H}}(\Omega)$ such that
\begin{align*}
	&\phi(\Xi(u))
	\l(\int_\Omega \left(|\nabla u|^{p-2}\nabla u +\mu(x) |\nabla u|^{q-2}\nabla u \right)\cdot \nabla \varphi\,\mathrm{d} x+\int_\Omega |u|^{p-2}u\varphi\, \mathrm{d}x\r)\\
	&= \int_{\Omega} f(x,u)\varphi\,\mathrm{d} x+\int_{\partial\Omega}g(x,u)\varphi\,\mathrm{d}\sigma
\end{align*}
is satisfied for all test functions $\varphi \in  W^{1,\mathcal{H}}(\Omega)$, where $\phi$ and $\Xi$ are given in \eqref{new-17} and \eqref{new-16}, respectively.

The main result in this paper is the following one.

\begin{theorem}\label{theorem_main_result}
	Let hypotheses \eqref{H1}--\eqref{H3} be satisfied. Then problem \eqref{problem} has a least energy sign-changing solution $y_0\in W^{1,\mathcal{H}}(\Omega)$.
\end{theorem}

The proof of Theorem \ref{theorem_main_result} uses variational methods along with minimization arguments. Precisely, the energy functional $\mathcal{E}\colon  W^{1,\mathcal{H}}(\Omega) \to \R$ associated to problem \eqref{problem} is defined by
\begin{equation*}
	\mathcal{E}(u) = \Phi[\Xi(u)] -\int_\Omega F(x,u)\,\mathrm{d} x-\int_{\partial\Omega}G(x,u)\,\mathrm{d}\sigma,
\end{equation*}
whereby $\Phi\colon  [0,\infty) \to [0,\infty)$ is given as
\begin{align*}
	\Phi(s)= \int_0^s\phi(\tau)\,\mathrm{d} \tau= as  +\frac{b}{\zeta}s^\zeta.
\end{align*}
The idea in finding a least energy sign-changing solution (also called nodal ground state solution) is to minimize $\mathcal{E}$ over the constraint set
\begin{align}\label{constraint-set}
	\mathcal{C}=\Big\{u \in  W^{1,\mathcal{H}}(\Omega)\colon u^{\pm}\neq 0,\, \l\langle \mathcal{E}'(u),u^+ \r\rangle= \l\langle \mathcal{E}'(u),-u^- \r\rangle=0 \Big\},
\end{align}
where $u^+=\max\{u,0\}$ and $u^-=\max\{-u,0\}$. It should be pointed out that the set $\mathcal{E}$ differs from the nodal Nehari manifold which is defined as
\begin{align}\label{nodal-Nehari-manifold}
	\mathcal{N}_0= \left\{ u \in  W^{1, \mathcal{H} } ( \Omega )\colon \pm u^\pm \in \mathcal{N} \right\}
\end{align}
with
\begin{align*}
	\mathcal{N} = \left\{ u \in  W^{1, \mathcal{H} } ( \Omega )\colon \langle \mathcal{E}'(u) , u \rangle = 0, \; u \neq 0 \right\}
\end{align*}
being the Nehari manifold to \eqref{problem}. Indeed, due to the nonlocal character of the problem in terms of a Kirchhoff function, we cannot split our energy functional in a positive and a negative part. Instead, we have for $u\in W^{1,\mathcal{H}}(\Omega)$ with $u^+ \neq 0 \neq u^-$ (so $u = u^+ - u^-$) and $\zeta>1$ the relations
\begin{equation}\label{splitting-energy-functional}
	\begin{aligned}
		\langle \mathcal{E}'(u),u^+\rangle&>\langle \mathcal{E}'(u^+),u^+\rangle,\quad
		\langle \mathcal{E}'(u),-u^-\rangle>\langle \mathcal{E}'(-u^-),-u^-\rangle, \\
		\mathcal{E}(u)&>\mathcal{E}(u^+)+\mathcal{E}(-u^-).
	\end{aligned}
\end{equation}
This makes the treatment of \eqref{problem} much more complicated as in the case $\phi\equiv 1$ since in that case we would have equations instead of strict inequalities in \eqref{splitting-energy-functional}.

Coming back to the constraint set $\mathcal{C}$ we know that all sign-changing solutions of \eqref{problem} are located in the set $\mathcal{C}$. That means that the global minimizer of $\mathcal{E}$ over the set $\mathcal{C}$ has least energy and is a nodal ground state solution if one can show that it is a critical point of $\mathcal{E}$. As already mentioned by Bartsch-Weth \cite{Bartsch-Weth-2005}, this is not a priori clear since the constraint set $\mathcal{C}$ is, in general, not a manifold anymore. However, if we set $\Phi\equiv 1$, $p=2$ and $\mu(x)\equiv 0$, then $\mathcal{C}\cap H$ with $H=H^1_0(\Omega)\cap H^2(\Omega)$ equipped with the scalar product from $H^2(\Omega)$, is a $C^1$-manifold of codimension two in $H$, see Bartsch-Weth \cite{Bartsch-Weth-2003}, but it is not complete in $H$ in general.

Our paper can be seen as an extension of the works of Gasi\'{n}ski-Winkert \cite{Gasinski-Winkert-2021} and Crespo-Blanco-Gasi\'{n}ski-Winkert \cite{Crespo-Blanco-Gasinski-Winkert-2023}. Indeed, in \cite{Gasinski-Winkert-2021} the existence of a sign-changing solution to the problem
\begin{equation*}%\label{problem}
	\begin{aligned}
		-\operatorname{div}\left(|\nabla u|^{p-2}\nabla u+\mu(x) |\nabla u|^{q-2}\nabla u\right) & =f(x,u)-|u|^{p-2}u-\mu(x)|u|^{q-2}u&& \text{in } \Omega,\\
		\left(|\nabla u|^{p-2}\nabla u+\mu(x) |\nabla u|^{q-2}\nabla u\right) \cdot \nu & = g(x,u) &&\text{on } \partial \Omega,
	\end{aligned}
\end{equation*}
has been shown by minimizing the corresponding energy functional over the nodal Nehari manifold given in \eqref{nodal-Nehari-manifold}. As mentioned above, this treatment is not possible in our setting due to the lack of the splitting of the energy functional related to problem \eqref{problem}, see again \eqref{splitting-energy-functional}. In \cite{Crespo-Blanco-Gasinski-Winkert-2023} the authors deal with multiplicity of solutions for Kirchhoff problems with Dirichlet boundary conditions and a double phase setting. We are going to combine the ideas from both papers, using a similar method as in \cite{Crespo-Blanco-Gasinski-Winkert-2023} and taking ideas from \cite{Gasinski-Winkert-2021} how to deal with the nonlinear boundary term.

The method of minimizing related energy functionals over the constraint set $\mathcal{C}$ was first used in the paper of Bartsch-Weth \cite{Bartsch-Weth-2003} (see also \cite{Bartsch-Weth-2005} of the same authors) in order to get a sign-changing solution for semilinear equations while additional properties like the Morse index have been proved as well. A nonlocal version of the problem in \cite{Bartsch-Weth-2003} and \cite{Bartsch-Weth-2005} has been studied by Shuai \cite{Shuai-2015} while Tang-Cheng \cite{Tang-Cheng-2016} were able to weaken the assumptions in \cite{Shuai-2015} and proved the existence of a nodal ground state solution to the Kirchhoff problem
\begin{align*}
	-\left(a+b \|\nabla u\|_2^2\right)\Delta u = f(u)\quad  \text{in } \Omega, \quad
	u = 0 \quad  \text{on } \partial\Omega.
\end{align*}
In Tang-Chen \cite{Tang-Chen-2017} and Wang-Zhang-Cheng \cite{Wang-Zhang-Cheng-2018}, Schr\"{o}dinger-Kirchhoff problems of the shape
\begin{align*}%\label{problem-Schroedinger-Kirchhoff-type}
	-\left(a+b \|\nabla u\|_2^2\right)\Delta u+V(x)u = f(u)\quad  \text{in } \R^3, \quad
	u \in H^1(\R^3)
\end{align*}
have been studied under different hypotheses on the data. In both papers the existence of a nodal ground state solution could be shown by using a similar constraint set as in \eqref{constraint-set}. Closely related in this direction are the works of Figueiredo-Santos J\'{u}nior \cite{Figueiredo-Santos-Junior-2015} for Schr\"{o}dinger-Kirchhoff equations with potential vanishing at infinity, Li-Wang-Zhang \cite{Li-Wang-Zhang-2020} for $p$-Laplacian Kirchhoff-type problems with logarithmic nonlinearity, Liang-R\u{a}dulescu \cite{Liang-Radulescu-2020} for fractional Kirchhoff problems with logarithmic and critical nonlinearity, Zhang \cite{Zhang-2021} for Schr\"{o}dinger-Kirchhoff-type problems and Zhang \cite{Zhang-2023} for $N$-Laplacian equations of Kirchhoff type.

It is worth mentioning that in all the papers stated above dealing with a nonlocal term (except the one in \cite{Crespo-Blanco-Gasinski-Winkert-2023}), the Kirchhoff function has the form
\begin{align*}
	s\mapsto a+bs,
\end{align*}
while we allow a much more general function
\begin{align*}
	s\mapsto a+b s^{\zeta-1} \quad\text{with }\zeta\geq 1.
\end{align*}
Furthermore, most of the papers we mentioned are dealing with the case $a>0$ while we are able to include the degenerate case which is more complicated to deal with it and has much more applications in mathematical physics, see above.

The paper is organized as follows. In Section \ref{Section_2} we present the main properties of Musielak-Orlicz Sobolev spaces including a very useful equivalent norm in $W^{1,\mathcal{H}}(\Omega)$ and we present some important tools needed in the sequel like the Poincar\'{e}-Miranda existence theorem and the quantitative deformation lemma. Finally, Section \ref{Section_3} is devoted to the proof of Theorem \ref{theorem_main_result} which is based on several auxiliary results in order to show that the minimizer of $\mathcal{E}$ over $\mathcal{C}$ is the required nodal ground state solution.

%***********************************************************************************
\section{Preliminaries}\label{Section_2}
%***********************************************************************************

In this section we mention the main properties of our function space as well as the double phase operator and we  recall some tools that we are going to use in Section \ref{Section_3} in order to prove Theorem \ref{theorem_main_result}. For more information about Musielak-Orlicz Sobolev spaces and the double phase operator we also refer to Colasuonno-Squassina \cite{Colasuonno-Squassina-2016}, Harjulehto-H\"{a}st\"{o} \cite{Harjulehto-Hasto-2019}, Liu-Dai \cite{Liu-Dai-2018} and Perera-Squassina \cite{Perera-Squassina-2018}.

Let $\Omega\subseteq \R^N$ be a bounded domain with a Lipschitz boundary $\partial\Omega$ and $N\geq 2$. Given $1\leq r \leq \infty$, $L^{r}(\Omega)$ and $L^r ( \Omega ; \R^N )$ stand for the usual Lebesgue spaces equipped with the norm $\| \cdot \|_r$. For $1<r<\infty$, we denote by $W^{1,r}(\Omega)$ the corresponding Sobolev space endowed with the equivalent norm
\begin{align*}
	\|u\|_{1,r}=\l(\|\nabla u\|_r^r+\|u\|_r^r \r)^{\frac{1}{r}}.
\end{align*}
Moreover, we denote by $L^{r}(\partial \Omega)$ the boundary Lebesgue space with norm $\|\cdot \|_{r,\partial\Omega}$ for any $r\in [1,\infty )$ given by
\begin{align*}
	\|u\|_{r,\partial\Omega}=\l(\int_{\partial\Omega} |u|^r\,\mathrm{d}\sigma\r)^{\frac{1}{r}},
\end{align*}
where $\sigma$ denotes the $(N-1)$-dimensional Hausdorff surface measure.

Denoting by $M(\Omega)$ the set of all measurable function $u \colon \Omega \to \R$ and supposing hypotheses \eqref{H1}, we introduce the function $\mathcal{H} \colon \Omega \times [0,\infty) \to [0,\infty)$ defined by
\begin{align*}
	\mathcal{H} ( x , t )= t^p + \mu(x) t^q.
\end{align*}
Then we can define the Musielak-Orlicz space $L^\mathcal{H} ( \Omega )$ by
\begin{align*}
	L^\mathcal{H} ( \Omega ) = \left \{ u \in M(\Omega)\colon \rho_{\mathcal{H}} ( u ) < +\infty \right \},
\end{align*}
which is equipped with the norm
\begin{align*}
	\| u \|_{\mathcal{H}} = \inf \left \{ \lambda >0 \colon \rho_{\mathcal{H}} \left ( \frac{u}{\lambda} \right ) \leq 1  \right \},
\end{align*}
whereby $\rho_{\mathcal{H}}$ is the modular function defined by
\begin{align*}%\label{modular}
	\rho_{\mathcal{H}}(u)=\int_\Omega \mathcal{H} ( x , |u| ) \,\mathrm{d} x=\int_\Omega \big( |u|^{p} + \mu(x) | u |^q \big ) \, \mathrm{d} x.
\end{align*}
In the sequel, we also have to deal with the seminormed space $L^q_\mu ( \Omega )$ given by
\begin{align*}
	L^q_\mu ( \Omega ) = \left \{ u \in M ( \Omega ) \colon \int_\Omega \mu(x) | u |^q \, \mathrm{d} x < +\infty \right \},
\end{align*}
while the related seminorm is defined as
\begin{align*}
	\| u \|_{q,\mu} = \left ( \int_\Omega \mu(x) | u |^q \, \mathrm{d} x \right)^{ \frac{1}{q} }.
\end{align*}
Similarly, we can introduce $L^q_\mu(\Omega;\R^N)$. The associated Musielak-Orlicz Sobolev space
\begin{align*}
	W^{1,\mathcal{H}}(\Omega)= \left \{ u \in L^\mathcal{H}(\Omega) \colon | \nabla u | \in L^{\mathcal{H}} ( \Omega ) \right\}
\end{align*}
is equipped with the norm
\begin{align*}
	\| u \|_{1 , \mathcal{H} } = \| \nabla u \|_{\mathcal{H}} + \| u \|_{\mathcal{H}}.
\end{align*}
It is well known that both spaces $L^{\mathcal{H}}(\Omega)$ and $W^{1,\mathcal{H}}(\Omega)$ are reflexive Banach spaces, see Crespo-Blanco-Gasi\'{n}ski-Harjulehto-Winkert \cite[Proposition 2.12]{Crespo-Blanco-Gasinski-Harjulehto-Winkert-2022}. Moreover, from Amoroso-Crespo-Blanco-Pucci-Winkert \cite[Proposition 3.1]{Amoroso-Crespo-Blanco-Pucci-Winkert-2023} (see also Crespo-Blanco-Papageorgiou-Winkert \cite[Proposition 2.2]{Crespo-Blanco-Papageorgiou-Winkert-2022}) we can equip $W^{1,\mathcal{H}}(\Omega)$ with the equivalent norm
\begin{align*}
	\|u\|=\inf\l\{\lambda>0\colon\int_\Omega \l(\l(\frac{|\nabla u|}{\lambda}\r)^p+\mu(x)\l(\frac{|\nabla u|}{\lambda}\r)^q\r)\,\mathrm{d} x+\int_\Omega \l(\frac{|u|}{\lambda}\r)^{p}\,\mathrm{d} x\leq 1\r\}.
\end{align*}
The corresponding modular $\rho$ to $\|\cdot\|$ is given by
\begin{align*}%\label{modular2}
	\rho(u)=\int_\Omega \big(|\nabla u|^p+\mu(x)|\nabla u|^q\big)\,\mathrm{d}x+\int_\Omega |u|^p\,\mathrm{d} x
\end{align*}
for $u \in W^{1,\mathcal{H}}(\Omega)$.

The following proposition can be found in the work of Amoroso-Crespo-Blanco-Pucci-Winkert \cite[Proposition 3.2]{Amoroso-Crespo-Blanco-Pucci-Winkert-2023}.

\begin{proposition}\label{proposition_modular_properties2}
	Let \eqref{H1} be satisfied, $\tau>0$ and $y\in W^{1,\mathcal{H}}(\Omega)$. Then the following hold:
	\begin{enumerate}
		\item[\textnormal{(i)}]
			If $y\neq 0$, then $\|y\|=\tau$ if and only if $ \rho(\frac{y}{\tau})=1$;
		\item[\textnormal{(ii)}]
			$\|y\|<1$ (resp.\,$>1$, $=1$) if and only if $ \rho(y)<1$ (resp.\,$>1$, $=1$);
		\item[\textnormal{(iii)}]
			If $\|y\|<1$, then $\|y\|^q\leq \rho(y)\leq\|y\|^p$;
		\item[\textnormal{(iv)}]
			If $\|y\|>1$, then $\|y\|^p\leq \rho(y)\leq\|y\|^q$;
		\item[\textnormal{(v)}]
			$\|y\|\to 0$ if and only if $ \rho(y)\to 0$;
		\item[\textnormal{(vi)}]
			$\|y\|\to +\infty$ if and only if $ \rho(y)\to +\infty$.
	\end{enumerate}
\end{proposition}

From Crespo-Blanco-Gasi\'{n}ski-Harjulehto-Winkert \cite[Proposition 2.16]{Crespo-Blanco-Gasinski-Harjulehto-Winkert-2022} we have the following embeddings which will be used later.

\begin{proposition}
	\label{proposition_embeddings}
	Let \eqref{H1} be satisfied. Then the following embeddings hold:
	\begin{enumerate}
		\item[\textnormal{(i)}]
			$ W^{1, \mathcal{H} } ( \Omega ) \hookrightarrow L^{r}(\Omega)$ is continuous for all $r \in [1,p^*]$ and compact for all $r \in [1,p^*)$.
		\item[\textnormal{(ii)}]
			$ W^{1, \mathcal{H} } ( \Omega ) \hookrightarrow L^{r}(\partial\Omega)$ is continuous for all $r \in [1,p_*]$ and compact for all $r \in [1,p_*)$.
	\end{enumerate}
\end{proposition}

Furthermore, for $s \in \R$, we set $s^{\pm}=\max\{\pm s,0\}$ and for a function $u \in  W^{1,\mathcal{H}}(\Omega)$ we define $u^{\pm}(\cdot)=u(\cdot)^{\pm}$. Clearly, $|u|=u^++u^-$ and $u=u^+-u^-$. We also know from Crespo-Blanco-Gasi\'{n}ski-Harjulehto-Winkert \cite[Proposition 2.17]{Crespo-Blanco-Gasinski-Harjulehto-Winkert-2022} that $u^{\pm} \in  W^{1,\mathcal{H}}(\Omega)$ whenever $u \in W^{1,\mathcal{H}}(\Omega)$. The Lebesgue measure of a set $V \subseteq \R^N$ will be denoted by $|V|_N$.

Next, we recall the main properties of our operator. For this purpose, let  $B\colon  W^{1, \mathcal{H} } ( \Omega )\to  W^{1, \mathcal{H} } ( \Omega )^*$ be defined by
\begin{align*}%\label{operator_representation}
	\langle B(u),v\rangle =\int_\Omega \big(|\nabla u|^{p-2}\nabla u+\mu(x)|\nabla u|^{q-2}\nabla u \big)\cdot\nabla v \,\mathrm{d} x+\int_\Omega |u|^{p-2}u v\,\mathrm{d}x
\end{align*}
for all $u,v\in W^{1, \mathcal{H} } ( \Omega )$. Here, $\langle\,\cdot\,,\,\cdot\,\rangle$ stands for the duality pairing between the space $ W^{1, \mathcal{H} } ( \Omega )$ and its dual space $ W^{1, \mathcal{H} } ( \Omega )^*$.  The following proposition has been proved in Amoroso-Crespo-Blanco-Pucci-Winkert \cite[Proposition 3.3]{Amoroso-Crespo-Blanco-Pucci-Winkert-2023}.

\begin{proposition}
	%\label{properties_operator_double_phase}
	Under hypotheses \eqref{H1}, the operator $B$ is bounded, continuous, strictly monotone and satisfies the \textnormal{(S$_+$)}-property, i.e., if
	\begin{align*}
		u_n\rightharpoonup u \quad \text{in } W^{1,\mathcal{H}}(\Omega) \quad\text{and}\quad  \limsup_{n\to\infty}\,\langle Bu_n,u_n-u\rangle\le 0
	\end{align*}
	hold, then we have $u_n\to u$ in $ W^{1,\mathcal{H}}(\Omega)$.
\end{proposition}

Finally, we recall some important tools that will be needed in the next section. The first one is the quantitative deformation lemma, see Willem \cite[Lemma 2.3]{Willem-1996}.

\begin{lemma} \label{Le:DeformationLemma}
	Let $X$ be a Banach space, $\mathcal{E} \in C^1(X;\R)$, $\emptyset \neq S \subseteq X$, $c \in \R$, $\varepsilon,\delta > 0$ such that for all $u \in \mathcal{E}^{-1}([c - 2\varepsilon, c + 2\varepsilon]) \cap S_{2 \delta}$ there holds $\| \mathcal{E}'(u) \|_* \geq 8\varepsilon / \delta$, where $S_{r} = \{ u \in X \colon d(u,S) = \inf_{u_0 \in S} \| u - u_0 \| < r \}$ for any $r > 0$.
	Then there exists $\eta \in C([0, 1] \times X; X)$ such that
	\begin{enumerate}
		\item[\textnormal{(i)}]
			$\eta (t, u) = u$, if $t = 0$ or if $u \notin \mathcal{E}^{-1}([c - 2\varepsilon, c + 2\varepsilon]) \cap S_{2 \delta}$;
		\item[\textnormal{(ii)}]
			$\mathcal{E}( \eta( 1, u ) ) \leq c - \varepsilon$ for all $u \in \mathcal{E}^{-1} ( ( - \infty, c + \varepsilon] ) \cap S $;
		\item[\textnormal{(iii)}]
			$\eta(t, \cdot )$ is an homeomorphism of $X$ for all $t \in [0,1]$;
		\item[\textnormal{(iv)}]
			$\| \eta(t, u) - u \| \leq \delta$ for all $u \in X$ and $t \in [0,1]$;
		\item[\textnormal{(v)}]
			$\mathcal{E}( \eta( \cdot , u))$ is decreasing for all $u \in X$;
		\item[\textnormal{(vi)}]
			$\mathcal{E}(\eta(t, u)) < c$ for all $u \in \mathcal{E}^{-1} ( ( - \infty, c] ) \cap S_\delta$ and $t \in (0, 1]$.
	\end{enumerate}
\end{lemma}

The second one is the so-called Poincar\'{e}-Miranda existence theorem, see Dinca-Mawhin \cite[Corollary 2.2.15]{Dinca-Mawhin-2021}.

\begin{lemma}
	\label{lemma-poincare-miranda}
	Let $Q = [-t_1, t_1] \times \cdots \times [-t_N, t_N]$ with $t_i > 0$ for $i =1,\ldots, N$ and $\varphi \colon Q \to \R^N$ be continuous with the component functions $\varphi_i\colon Q\to\R$ for $i=1,\ldots, N$. If, for each $i =1,\ldots, N$, one has
	\begin{align*}
		\begin{aligned}
			\varphi_i (u) & \leq 0 \quad\text{when } u \in Q \text{ and } u_i = -t_i, \\
			\varphi_i (u) & \geq 0 \quad\text{when } u \in Q \text{ and } u_i = t_i,
		\end{aligned}
	\end{align*}
	then $\varphi$ has at least one zero point in $Q$.
\end{lemma}

%***********************************************************************************
\section{Least energy sign-changing solution}\label{Section_3}
%***********************************************************************************
In this section we are going to prove Theorem \ref{theorem_main_result} about the existence of a least energy sign-changing solution to problem \eqref{problem}. First, we recall the energy functional $\mathcal{E}\colon  W^{1,\mathcal{H}}(\Omega) \to \R$ related to problem \eqref{problem} which is given by
\begin{equation*}%\label{Energy}
	\mathcal{E}(u) = \Phi[\Xi(u)] -\int_\Omega F(x,u)\,\mathrm{d} x-\int_{\partial\Omega}G(x,u)\,\mathrm{d}\sigma,
\end{equation*}
where $\Phi\colon  [0,\infty) \to [0,\infty)$ is defined by
\begin{align*}
	\Phi(s)= \int_0^s\phi(\tau)\,\mathrm{d} \tau= as  +\frac{b}{\zeta}s^\zeta.
\end{align*}
As already mentioned in the Introduction, we will consider the following constraint set
\begin{align*}
	\mathcal{C}=\Big\{u \in  W^{1,\mathcal{H}}(\Omega)\colon u^{\pm}\neq 0,\, \l\langle \mathcal{E}'(u),u^+ \r\rangle= \l\langle \mathcal{E}'(u),-u^- \r\rangle=0 \Big\}.
\end{align*}

Next, we prove several auxiliary results.

\begin{proposition}
	\label{proposition_nodal_unique-pair}
	Let hypotheses \eqref{H1}--\eqref{H3} be satisfied and let $u \in  W^{1,\mathcal{H}}(\Omega)$ be such that $u^{\pm}\neq 0$. Then we can find a unique pair of positive numbers $(\alpha_u,\beta_u)$ such that $\alpha_uu^+-\beta_uu^-\in \mathcal{C}$. Furthermore, if $u\in \mathcal{C}$, we have
	\begin{align*}
		\mathcal{E}(t_1u^+-t_2u^-)\leq \mathcal{E}(u^+-u^-)=\mathcal{E}(u)
	\end{align*}
	for all $t_1,t_2\geq 0$ with strict inequality whenever $(t_1,t_2)\neq (1,1)$. Moreover, for $u\in\mathcal{C}$ we have the following sign information:
	\begin{enumerate}
		\item[\textnormal{(i)}]
			if $\alpha>1$ and $0 < \beta \leq \alpha$, then $ \langle \mathcal{E}' (\alpha u^+-\beta u^-),\alpha u^+\rangle < 0$;
		\item[\textnormal{(ii)}]
			if $\alpha<1$ and $0 < \alpha \leq \beta$, then $ \langle \mathcal{E}' (\alpha u^+-\beta u^-),\alpha u^+\rangle > 0$;
		\item[\textnormal{(iii)}]
			if $\beta>1$ and $0 < \alpha \leq \beta$, then $ \langle \mathcal{E}' (\alpha u^+-\beta u^-),-\beta u^-\rangle < 0$;
		\item[\textnormal{(iv)}]
			if $\beta<1$ and $0 < \beta \leq \alpha$, then $ \langle \mathcal{E}' (\alpha u^+-\beta u^-),-\beta u^-\rangle > 0$.
	\end{enumerate}
\end{proposition}

\begin{proof}
	Let $u \in  W^{1,\mathcal{H}}(\Omega)$ with $u^{\pm}\neq 0$. The proof is divided into several parts.

	{\bf Part 1:} We are going to prove the existence of a pair $(\alpha_u,\beta_u) \in (0,\infty)\times (0,\infty)$ such that $\alpha_uu^+-\beta_uu^-\in \mathcal{C}$.

	First, for $t\in (0,1)$ and for $|s|>0$, by taking \eqref{H3}\eqref{H3v} into account, we have
	\begin{align}\label{new-1}
		\frac{f(x,ts)(ts)}{t^{q\zeta}|s|^{q\zeta}} \leq \frac{f(x,s)s}{|s|^{q\zeta}} \quad \text{for a.a.\,}x\in \Omega.
	\end{align}
	Analogously, for $t\in (0,1)$ and for $|s|>0$, again by \eqref{H3}\eqref{H3v}, it follows that
	\begin{align}\label{new-2}
		\frac{g(x,ts)(ts)}{t^{q\zeta}|s|^{q\zeta}} \leq \frac{g(x,s)s}{|s|^{q\zeta}} \quad \text{for a.a.\,}x\in \partial\Omega.
	\end{align}
	From \eqref{new-1} and \eqref{new-2} we get the inequalities
	\begin{align}
		f(x,ts)s &\leq t^{q\zeta-1}f(x,s)s \quad \text{for a.a.\,}x\in \Omega,\label{n4}\\
		g(x,ts)s &\leq t^{q\zeta-1}g(x,s)s \quad \text{for a.a.\,}x\in \partial\Omega.\label{n5}
	\end{align}
	Then, taking \eqref{n4} and \eqref{n5} into account, we obtain
	\begin{align*}
		& \l\langle \mathcal{E}'\l(\alpha u^+-\beta u^-\r),\alpha u^+\r\rangle  \\
		& =\Bigg( a+b\bigg(\frac{1}{p}\l\|\alpha u^+-\beta u^- \r\|_{1,p}^p
		+\frac{1}{q}\l\|\nabla \l(\alpha u^+-\beta u^-\r) \r\|_{q,\mu}^q\bigg)^{\zeta-1}\Bigg)\\
		&\quad \times \l(\l\| \alpha u^+ \r\|_{1,p}^p+\l\|\nabla \l(\alpha u^+\r) \r\|_{q,\mu}^q\r)\\
		& \quad -\int_\Omega  f\l(x,\alpha u^+\r)\alpha u^+\,\mathrm{d} x-\int_{\partial\Omega}  g\l(x,\alpha u^+\r)\alpha u^+\,\mathrm{d} \sigma  \\
		& \geq \frac{b}{p^{\zeta-1}} \alpha^{p\zeta} \l\| u^+\r\|_{1,p}^{p\zeta}-\alpha^{q\zeta}\int_\Omega  f\l(x, u^+\r) u^+\,\mathrm{d} x-\alpha^{q\zeta}\int_{\partial\Omega}  g\l(x, u^+\r) u^+\,\mathrm{d} \sigma>0,
	\end{align*}
	whenever $\alpha>0$ is small enough and for all $\beta\geq 0$ since $p<q$ by \eqref{H1}. In the same way, for $\beta>0$ small enough and for all $\alpha\geq 0$, we have
	\begin{align*}
		& \l\langle \mathcal{E}'\l(\alpha u^+-\beta u^-\r),-\beta u^-\r\rangle   \\
		& =\Bigg(a+b\bigg(\frac{1}{p}\l\|\alpha u^+-\beta u^- \r\|_{1,p}^p+\frac{1}{q}\l\|\nabla \l(\alpha u^+-\beta u^-\r) \r\|_{q,\mu}^q\bigg)^{\zeta-1}\Bigg)\\
		&\quad \times \l(\l\|\beta u^- \r\|_{1,p}^p+\l\|\nabla \l(\beta u^-\r) \r\|_{q,\mu}^q\r)\\
		& \quad -\int_\Omega  f\l(x,-\beta u^-\r)\l(-\beta u^-\r)\,\mathrm{d} x -\int_{\partial\Omega}  g\l(x,-\beta u^-\r)\l(-\beta u^-\r)\,\mathrm{d} \sigma\\
		& \geq \frac{b}{p^{\zeta-1}} \beta^{p\zeta} \l\| u^-\r\|_{1,p}^{p\zeta}-\beta^{q\zeta}\int_\Omega  f\l(x, -u^-\r) \l(-u^-\r)\,\mathrm{d} x\\
		&\quad-\beta^{q\zeta}\int_{\partial \Omega}  g\l(x, -u^-\r) \l(-u^-\r)\,\mathrm{d} \sigma>0.
	\end{align*}
	Therefore, we can find a number $\eta_1>0$ such that
	\begin{align}\label{G1}
		 & \l\langle \mathcal{E}'\l(\eta_1 u^+-\beta u^-\r),\eta_1 u^+\r\rangle>0
		\quad\text{and}\quad
		\l\langle\mathcal{E}'\l(\alpha u^+-\eta_1 u^-\r),-\eta_1 u^-\r\rangle>0
	\end{align}
	for all $\alpha,\beta\geq 0$. Choosing $\eta_2>\max\{1,\eta_1\}$ large enough, by applying \eqref{H3}\eqref{H3ii}, it holds for $\beta \in [0,\eta_2]$
	\begin{align*}
		& \frac{\l\langle \mathcal{E}'\l(\eta_2 u^+-\beta u^-\r),\eta_2 u^+\r\rangle}{\eta_2^{q\zeta}}   \\
		& =\frac{a \l(\l\|\eta_2 u^+ \r\|_{1,p}^p+\l\| \nabla \l(\eta_2 u^+\r) \r\|_{q,\mu}^q\r)}{\eta_2^{q\zeta}}  \\
		 & \quad +\frac{b\l(\frac{1}{p}\l\|\eta_2 u^+-\beta u^- \r\|_{1,p}^p+\frac{1}{q}\l\|\nabla \l(\eta_2 u^+-\beta u^-\r) \r\|_{q,\mu}^q\r)^{\zeta-1}
		}{\eta_2^{q\zeta}}  \\
		& \quad\times \l(\l\|\eta_2 u^+\r\|_{1,p}^p+\l\| \nabla \l(\eta_2 u^+\r) \r\|_{q,\mu}^q\r)\\
		& \quad -\int_\Omega  \frac{f\l(x,\eta_2 u^+\r)\eta_2 u^+}{\eta_2^{q\zeta}}\,\mathrm{d} x-\int_{\partial\Omega}  \frac{g\l(x,\eta_2 u^+\r)\eta_2 u^+}{\eta_2^{q\zeta}}\,\mathrm{d} \sigma  \\
		& \leq a \frac{1}{\eta_2^{q(\zeta-1)}}\l(\l\|u^+ \r\|_{1,p}^p+\l\| \nabla  u^+ \r\|_{q,\mu}^q\r) \\
		& \quad +b
		\l(\l\|u^+- u^- \r\|_{1,p}^p+\l\|\nabla \l( u^+- u^-\r) \r\|_{q,\mu}^q\r)^{\zeta}   \\
		& \quad -\int_{\Omega_+}  \frac{f\l(x,\eta_2 u^+\r)}{\l(\eta_2u^+\r)^{q\zeta-1}}\l(u^+\r)^{q\zeta}\,\mathrm{d} x-\int_{(\partial \Omega)_+}  \frac{g\l(x,\eta_2 u^+\r)}{\l(\eta_2u^+\r)^{q\zeta-1}}\l(u^+\r)^{q\zeta}\,\mathrm{d} \sigma<0,\\
	\end{align*}
	where
	\begin{align*}
		\Omega_+&:=\left\{x\in\Omega\colon u(x)>0\text{ for a.a.\,}x\in\Omega\right\},\\
		\Omega_-&:=\left\{x\in\Omega\colon u(x)<0\text{ for a.a.\,}x\in\Omega\right\},\\
		(\partial\Omega)_+&:=\left\{x\in\partial\Omega\colon u(x)>0\text{ for a.a.\,}x\in\partial\Omega\right\},\\
		(\partial\Omega)_-&:=\left\{x\in\partial\Omega\colon u(x)<0\text{ for a.a.\,}x\in\partial\Omega\right\}.
	\end{align*}

	Similarly, for $\alpha \in [0,\eta_2]$ and $\eta_2>\max\{1,\eta_1\}$ large enough, we are able to prove that
	\begin{align*}
		\frac{\l\langle \mathcal{E}'\l(\alpha u^+-\eta_2 u^-\r),-\eta_2 u^-\r\rangle}{\eta_2^{q\zeta}}<0.
	\end{align*}
	Then it follows that
	\begin{align}\label{G2}
		\l\langle \mathcal{E}'\l( \eta_2 u^+-\beta u^-\r),\eta_2 u^+\r\rangle<0
		\quad\text{and}\quad
		\l\langle \mathcal{E}'\l(\alpha u^+-\eta_2 u^-\r),-\eta_2 u^-\r\rangle<0
	\end{align}
	for $\eta_2>\max\{1,\eta_1\}$ large enough and for $\alpha,\beta \in [0,\eta_2]$. Next, let $\varrho_u \colon [0,\infty) \times [0,\infty) \to \R^2$ be defined by
	\begin{equation*}
		\varrho_u (\alpha,\beta)
		= \left( \l\langle \mathcal{E}'\l( \alpha u^+-\beta u^-\r), \alpha u^+\r\rangle ,
		\l\langle \mathcal{E}'\l(\alpha u^+-\beta u^-\r),-\beta u^-\r\rangle \right).
	\end{equation*}
	From Lemma \ref{lemma-poincare-miranda} and \eqref{G1} as well as \eqref{G2} we find a pair $(\alpha_u,\beta_u)\in [\eta_1,\eta_2] \times [\eta_1,\eta_2] \subseteq (0,\infty)\times (0,\infty)$ satisfying $\varrho_u (\alpha_u,\beta_u) = (0,0)$. Hence $\alpha_uu^+-\beta_uu^-\in \mathcal{C}$. This proves Part 1.

	{\bf Part 2:} In this part we will prove the sign information given in (i)--(iv).

	To this end, let $u \in \mathcal{C}$. By the definition of $\mathcal{C}$, we have
	\begin{equation}\label{G3}
		\begin{aligned}
			0
			&=\l\langle \mathcal{E}'(u),u^+\r\rangle\\
			&=\l(a +b\l(\frac{1}{p}\l\|u\r\|_{1,p}^p+\frac{1}{q}\l\|\nabla u\r\|_{q,\mu}^q\r)^{\zeta-1}\r)
			\l(\l\| u^+\r\|_{1,p}^p+\l\|\nabla u^+\r\|_{q,\mu}^q\r)\\
			&\quad -\int_\Omega  f\l(x,u^+\r)u^+\,\mathrm{d} x-\int_{\partial \Omega}  g\l(x,u^+\r)u^+\,\mathrm{d} \sigma
		\end{aligned}
	\end{equation}
	and
	\begin{equation}\label{G4}
		\begin{aligned}
			0
			&=\l\langle \mathcal{E}'(u),-u^-\r\rangle\\
			&=\l(a +b\l(\frac{1}{p}\l\|u\r\|_{1,p}^p+\frac{1}{q}\l\|\nabla u\r\|_{q,\mu}^q\r)^{\zeta-1}\r)
			\l(\l\| u^-\r\|_{1,p}^p+\l\|\nabla u^-\r\|_{q,\mu}^q\r)\\
			&\quad -\int_\Omega  f\l(x,-u^-\r)\l(-u^-\r)\,\mathrm{d} x-\int_{\partial \Omega}  g\l(x,-u^-\r)\l(-u^-\r)\,\mathrm{d} \sigma.
		\end{aligned}
	\end{equation}
	We start with case \textnormal{(i)}. Let $\alpha>1$, $0 < \beta \leq \alpha$ and suppose by contradiction that
	\begin{equation}
		\label{G5}
		\begin{aligned}
			0
			&\leq\l\langle \mathcal{E}'\l(\alpha u^+-\beta u^-\r),\alpha u^+\r\rangle\\
			&=\l(a +b\l(\frac{1}{p}\l\|\alpha u^+-\beta u^-\r\|_{1,p}^p+\frac{1}{q}\l\|\nabla \l(\alpha u^+-\beta u^-\r)\r\|_{q,\mu}^q\r)^{\zeta-1}\r)\\
			&\quad\times\l(\l\|\alpha u^+\r\|_{1,p}^p+\l\|\nabla \l(\alpha u^+\r)\r\|_{q,\mu}^q\r)\\
			&\quad -\int_\Omega  f\l(x,\alpha u^+\r)\alpha u^+\,\mathrm{d} x-\int_{\partial \Omega}  g\l(x,\alpha u^+\r)\alpha u^+\,\mathrm{d} \sigma.
		\end{aligned}
	\end{equation}
	From \eqref{G5}, by using the fact that $\alpha>1$, it follows
	\begin{equation}
		\label{G7}
		\begin{aligned}
			0
			&\leq \l(a\alpha^q +b\alpha^{q\zeta}\l(\frac{1}{p}\l\| u\r\|_{1,p}^p+\frac{1}{q}\l\|\nabla u\r\|_{q,\mu}^q\r)^{\zeta-1}\r)
			\l(\l\| u^+\r\|_{1,p}^p+\l\|\nabla u^+\r\|_{q,\mu}^q\r)\\
			&\quad -\int_\Omega  f\l(x,\alpha u^+\r)\alpha u^+\,\mathrm{d} x-\int_{\partial \Omega}  g\l(x,\alpha u^+\r)\alpha u^+\,\mathrm{d} \sigma.
		\end{aligned}
	\end{equation}
	Then we divide \eqref{G7} by $\alpha^{q\zeta}$ and combine it with \eqref{G3}. This yields
	\begin{equation}
		\label{G8}
		\begin{aligned}
			&\int_{\Omega_+}  \l(\frac{f\l(x,\alpha u^+\r)}{\l(\alpha u^+\r)^{q\zeta-1}}-\frac{f\l(x, u^+\r)}{\l( u^+\r)^{q\zeta-1}}\r)\l(u^+\r)^{q\zeta}\,\mathrm{d} x\\
			&\quad+\int_{(\partial \Omega)_+}  \l(\frac{g\l(x,\alpha u^+\r)}{\l(\alpha u^+\r)^{q\zeta-1}}-\frac{g\l(x, u^+\r)}{\l( u^+\r)^{q\zeta-1}}\r)\l(u^+\r)^{q\zeta}\,\mathrm{d} \sigma\\
			&\leq a\l(\frac{1}{\alpha^{q(\zeta-1)}}-1\r)
			\l(\l\| u^+\r\|_{1,p}^p+\l\|\nabla u^+\r\|_{q,\mu}^q\r).
		\end{aligned}
	\end{equation}
	But this is a contradiction since the right-hand side of \eqref{G8} is nonpositive and the left-hand side is strictly positive due to  \eqref{H3}\eqref{H3v}. The case \textnormal{(ii)} works similarly. Indeed, if $\alpha<1$, $0 < \alpha \leq \beta$ and \eqref{G5} has the opposite sign, we can do it as above and get
	\eqref{G7} as well as \eqref{G8} in the opposite direction, respectively. But this is again a contradiction since in that case the right-hand side is nonnegative but the left-hand side is strictly negative, again because of \eqref{H3}\eqref{H3v}. Let us consider the case \textnormal{(iii)}. So, let $\beta>1$ and $0 < \alpha \leq \beta$ and suppose that
	\begin{equation}
		\label{G6}
		\begin{aligned}
			0
			&\leq\l\langle \mathcal{E}'\l(\alpha u^+-\beta u^-\r),-\beta u^-\r\rangle\\
			&=\l(a +b\l(\frac{1}{p}\l\|\alpha u^+-\beta u^-\r\|_{1,p}^p+\frac{1}{q}\l\|\nabla \l(\alpha u^+-\beta u^-\r)\r\|_{q,\mu}^q\r)^{\zeta-1}\r)\\
			&\quad\times\l(\l\|\beta u^-\r\|_{1,p}^p+\l\|\nabla \l(\beta u^-\r)\r\|_{q,\mu}^q\r)\\
			&\quad -\int_\Omega  f\l(x,-\beta u^-\r)\l(-\beta u^-\r)\,\mathrm{d} x-\int_{\partial \Omega}  g\l(x,-\beta u^-\r)\l(-\beta u^-\r)\,\mathrm{d} \sigma.
		\end{aligned}
	\end{equation}
	Taking $\beta>1$ into account, it follows from \eqref{G6} that
	\begin{equation}
		\label{G9}
		\begin{aligned}
			0
			&\leq \l(a\beta^q +b\beta^{q\zeta}\l(\frac{1}{p}\l\| u\r\|_{1,p}^p+\frac{1}{q}\l\|\nabla u\r\|_{q,\mu}^q\r)^{\zeta-1}\r)
			\l(\l\| u^-\r\|_{1,p}^p+\l\|\nabla u^-\r\|_{q,\mu}^q\r)\\
			&\quad -\int_\Omega  f\l(x,-\beta u^-\r)\l(-\beta u^-\r)\,\mathrm{d} x
			-\int_{\partial \Omega}  g\l(x,-\beta u^-\r)\l(-\beta u^-\r)\,\mathrm{d} \sigma.
		\end{aligned}
	\end{equation}
	Now we divide \eqref{G9} by $\beta^{q\zeta}$ and combine it with \eqref{G4} in order to get
	\begin{equation}
		\label{G10}
		\begin{aligned}
			&\int_{\Omega_-}  \l(\frac{f\l(x,-u^-\r)}{\l(u^-\r)^{q\zeta-1}}-\frac{f\l(x,-\beta u^-\r)}{\l(\beta u^-\r)^{q\zeta-1}}\r)\l(u^-\r)^{q\zeta}\,\mathrm{d} x\\
			&\quad+\int_{(\partial \Omega)_-}  \l(\frac{g\l(x,-u^-\r)}{\l(u^-\r)^{q\zeta-1}}-\frac{g\l(x,-\beta u^-\r)}{\l(\beta u^-\r)^{q\zeta-1}}\r)\l(u^-\r)^{q\zeta}\,\mathrm{d} \sigma\\
			&\leq a\l(\frac{1}{\beta^{q(\zeta-1)}}-1\r)
			\l(\l\| u^-\r\|_{1,p}^p+\l\|\nabla u^-\r\|_{q,\mu}^q\r).
		\end{aligned}
	\end{equation}
	Due to \eqref{H3}\eqref{H3v}, the left-hand side of \eqref{G10} is strictly positive and the right-hand side of \eqref{G10} is nonpositive, so a contradiction. Finally, for the case \textnormal{(iv)}, let $\beta<1$ and $0 < \beta \leq \alpha$ and suppose via contradiction that \eqref{G6} is satisfied in the opposite direction. Arguing analogously as above, this gives \eqref{G9} in the opposite direction, which implies \eqref{G10} in the opposite direction. So we have again a contradiction due to the fact that the right-hand side is nonnegative and, because of \eqref{H3}\eqref{H3v}, the left-hand side is strictly negative. Part 2 is proved.

	{\bf Part 3:} In this part we claim that the pair $(\alpha_u,\beta_u)$ from Part 1 is indeed unique.

	First, we suppose that $u\in \mathcal{C}$. We have to show that $(\alpha_u,\beta_u)=(1,1)$ is the unique pair of numbers such that $\alpha_u u^+-\beta_u u^-\in \mathcal{C}$ is satisfied. Indeed, let $(\alpha_0,\beta_0)\in (0,\infty)\times (0,\infty)$ be such that $\alpha_0 u^+-\beta_0 u^-\in \mathcal{C}$ holds. If $0<\alpha_0\leq\beta_0$, the cases \textnormal{(ii)} and \textnormal{(iii)} from Part 2 give us $1 \leq \alpha_0\leq\beta_0 \leq 1$ meaning that $\alpha_0=\beta_0=1$. Also if $0<\beta_0\leq\alpha_0$, then the cases \textnormal{(i)} and \textnormal{(iv)} from Part 2 imply that $1 \leq \beta_0\leq\alpha_0 \leq 1$ and so $\alpha_0=\beta_0=1$.

	Now we suppose that $u\not\in \mathcal{C}$ with $u^{\pm}\neq0$. Arguing by contradiction, assume that there exist two pairs $(\alpha_1,\beta_1)$, $(\alpha_2,\beta_2)$ of positive numbers $\alpha_1, \alpha_2, \beta_1, \beta_2$ such that
	\begin{align*}
		\vartheta_1:=\alpha_1 u^+-\beta_1 u^-\in \mathcal{C}
		\quad\text{and}\quad
		\vartheta_2:=\alpha_2 u^+-\beta_2 u^-\in \mathcal{C}.
	\end{align*}
	Then we obtain
	\begin{align}
		\label{G11}
		\vartheta_2=\l(\frac{\alpha_2}{\alpha_1}\r)\alpha_1u^+-\l(\frac{\beta_2}{\beta_1}\r)\beta_1u^-
		=\l(\frac{\alpha_2}{\alpha_1}\r)\vartheta_1^+-\l(\frac{\beta_2}{\beta_1}\r)\vartheta_1^- \in \mathcal{C}.
	\end{align}
	As $\vartheta_1\in \mathcal{C}$, the first case of Part 3 guarantees that the pair $(1,1)$ satisfying $1\cdot \vartheta_1^+-1\cdot \vartheta_1^-\in \mathcal{C}$ is unique. Using this fact along with \eqref{G11} leads to
	\begin{align*}
		\frac{\alpha_2}{\alpha_1}=\frac{\beta_2}{\beta_1}=1.
	\end{align*}
	But this implies $\alpha_1=\alpha_2$ and $\beta_1=\beta_2$. Therefore, the uniqueness is shown.

	{\bf Part 4:} We are going to prove that the unique pair $(\alpha_u,\beta_u)$ from Part 1 and 3 is the unique maximum point of the function $\Lambda_u\colon [0,\infty)\times [0,\infty)\to \R$ defined by
	\begin{align*}
		\Lambda_u(\alpha,\beta)=\mathcal{E}\l(\alpha u^+-\beta u^-\r).
	\end{align*}

	The idea is as follows: We show first that $\Lambda_u$ has a maximum and then we will prove that it cannot be achieved at a boundary point of $[0,\infty)\times [0,\infty)$. This gives us the assertion of Part 4.

	To this end, let $\alpha, \beta \geq 1$ and suppose, without any loss of generality, that $\alpha \geq \beta\geq 1$. We obtain
	\begin{equation}
		\label{G12}
		\begin{aligned}
			&\frac{\Lambda_u(\alpha,\beta)}{\alpha^{q\zeta}}\\
			&=\frac{\mathcal{E}(\alpha u^+-\beta u^-)}{\alpha^{q\zeta}}\\
			&=\frac{a\l(\frac{1}{p}\l\|\alpha u^+-\beta u^-\r\|_{1,p}^p+\frac{1}{q}\l\|\nabla \l(\alpha u^+-\beta u^-\r)\r\|_{q,\mu}^q\r)}{\alpha^{q\zeta}}\\
			&\quad +\frac{\frac{b}{\zeta}\l(\frac{1}{p}\l\|\alpha u^+-\beta u^-\r\|_{1,p}^p+\frac{1}{q}\l\|\nabla \l(\alpha u^+-\beta u^-\r)\r\|_{q,\mu}^q\r)^\zeta}{\alpha^{q\zeta}}\\
			&\quad -\int_\Omega  \frac{F\l(x,\alpha u^+-\beta u^-\r)}{\alpha^{q\zeta}}\,\mathrm{d} x
			-\int_{\partial \Omega}  \frac{G\l(x,\alpha u^+-\beta u^-\r)}{\alpha^{q\zeta}}\,\mathrm{d} \sigma\\
			&\leq \frac{1}{\alpha^{q(\zeta-1)}} a\l(\frac{1}{p}\l\| u\r\|_{1,p}^p+\frac{1}{q}\l\|\nabla u\r\|_{q,\mu}^q\r)
			+\frac{b}{\zeta}\l(\frac{1}{p}\l\| u\r\|_{1,p}^p+\frac{1}{q}\l\|\nabla u\r\|_{q,\mu}^q\r)^\zeta\\
			&\quad -\int_{\Omega_+}  \frac{F\l(x,\alpha u^+\r)}{\l(\alpha u^+\r)^{q\zeta}}\l(u^+\r)^{q\zeta}\,\mathrm{d} x-\int_{\Omega_-}  \frac{F\l(x,-\beta u^-\r)}{\l|-\beta u^-\r|^{q\zeta}}\frac{\beta^{q\zeta} \l(u^-\r)^{q\zeta}}{\alpha^{q\zeta}}\,\mathrm{d} x\\
			&\quad -\int_{(\partial \Omega)_+}  \frac{G\l(x,\alpha u^+\r)}{\l(\alpha u^+\r)^{q\zeta}}\l(u^+\r)^{q\zeta}\,\mathrm{d} \sigma-\int_{(\partial \Omega)_-}  \frac{G\l(x,-\beta u^-\r)}{\l|-\beta u^-\r|^{q\zeta}}\frac{\beta^{q\zeta} \l(u^-\r)^{q\zeta}}{\alpha^{q\zeta}}\,\mathrm{d} \sigma.
		\end{aligned}
	\end{equation}
	In consideration of \eqref{superlinear-F}, we have the following statements
	\begin{equation}\label{new-3}
		\begin{aligned}
			&\lim_{\alpha\to \infty}\l(-\int_{\Omega_+}  \frac{F\l(x,\alpha u^+\r)}{\l(\alpha u^+\r)^{q\zeta}}\l(u^+\r)^{q\zeta}\,\mathrm{d} x\r) = -\infty,\\
			&\limsup_{\substack{\beta\to \infty \\ \alpha\geq \beta}}\l(-\int_{\Omega_-}  \frac{F\l(x,-\beta u^-\r)}{\l|-\beta u^-\r|^{q\zeta}}\frac{\beta^{q\zeta} \l(u^-\r)^{q\zeta}}{\alpha^{q\zeta}}\,\mathrm{d} x\r)\leq 0,\\
			&\lim_{\alpha\to \infty}\l(-\int_{(\partial \Omega)_+}  \frac{G\l(x,\alpha u^+\r)}{\l(\alpha u^+\r)^{q\zeta}}\l(u^+\r)^{q\zeta}\,\mathrm{d} \sigma\r) = -\infty,\\
			&\limsup_{\substack{\beta\to \infty \\ \alpha\geq \beta}}\l(-\int_{(\partial \Omega)_-}  \frac{G\l(x,-\beta u^-\r)}{\l|-\beta u^-\r|^{q\zeta}}\frac{\beta^{q\zeta} \l(u^-\r)^{q\zeta}}{\alpha^{q\zeta}}\,\mathrm{d} \sigma\r)\leq 0.
		\end{aligned}
	\end{equation}
	Combining \eqref{G12} and \eqref{new-3} yields
	\begin{align*}%\label{new-5}
		\lim_{|(\alpha,\beta)|\to\infty}\Lambda_u(\alpha,\beta)=-\infty.
	\end{align*}
	This shows that $\Lambda_u$ has a maximum.

	Next, we are going to prove that a maximum point of $\Lambda_u$ cannot be achieved on the boundary of $[0,\infty)\times [0,\infty)$. For this purpose, assume via contradiction that $(0,\beta_0)$ is a maximum point of $\Lambda_u$ with $\beta_0\geq 0$. Then we have for $\alpha> 0$
	\begin{align*}
		\Lambda_u(\alpha,\beta_0)
		 & =\mathcal{E}(\alpha u^+-\beta_0 u^-)\\
		 & =a\l(\frac{1}{p}\l\|\alpha u^+-\beta_0 u^-\r\|_{1,p}^p+\frac{1}{q}\l\|\nabla \l(\alpha u^+-\beta_0 u^-\r)\r\|_{q,\mu}^q\r)\\
		 & \quad +\frac{b}{\zeta}\l(\frac{1}{p}\l\|\alpha u^+-\beta_0 u^-\r\|_{1,p}^p+\frac{1}{q}\l\|\nabla \l(\alpha u^+-\beta_0 u^-\r)\r\|_{q,\mu}^q\r)^\zeta \\
		 & \quad -\int_\Omega  F\l(x,\alpha u^+-\beta_0 u^-\r)\,\mathrm{d} x
		 -\int_{\partial \Omega}  G\l(x,\alpha u^+-\beta_0 u^-\r)\,\mathrm{d} \sigma
	\end{align*}
	and
	\begin{align*}
		\frac{\partial \Lambda_u(\alpha,\beta_0)}{\partial \alpha}
		 & =a\alpha^{p-1} \l\| u^+\r\|_{1,p}^p+a\alpha^{q-1}\l\|\nabla u^+\r\|_{q,\mu}^q \\
		 & \quad +b\l(\frac{1}{p}\l\|\alpha u^+-\beta_0 u^-\r\|_{1,p}^p+\frac{1}{q}\l\|\nabla \l(\alpha u^+-\beta_0 u^-\r)\r\|_{q,\mu}^q\r)^{\zeta-1} \\
		 & \quad\times\l(\alpha^{p-1} \l\| u^+\r\|_{1,p}^p+\alpha^{q-1}\l\|\nabla u^+\r\|_{q,\mu}^q\r) \\
		 & \quad -\int_\Omega  f\l(x,\alpha u^+\r)u^+\,\mathrm{d} x-\int_{\partial \Omega}  g\l(x,\alpha u^+\r)u^+\,\mathrm{d} \sigma.
	\end{align*}
	Therefore we obtain
	\begin{equation}\label{new-4}
		\begin{aligned}
			\frac{\partial \Lambda_u(\alpha,\beta_0)}{\partial \alpha}
			&\geq \frac{b}{p^{\zeta-1}}\alpha^{p\zeta-1} \l\| u^+\r\|_{1,p}^{p\zeta}-\int_\Omega  f\l(x,\alpha u^+\r)u^+\,\mathrm{d} x\\
			&\quad -\int_{\partial \Omega}  g\l(x,\alpha u^+\r)u^+\,\mathrm{d} \sigma.
		\end{aligned}
	\end{equation}
	Then we divide \eqref{new-4} by $\alpha^{p\zeta-1}>0$ which leads to
	\begin{align*}
		\frac{1}{\alpha^{p\zeta-1}}\dfrac{\partial \Lambda_u(\alpha,\beta_0)}{\partial \alpha}
		&\geq \frac{b}{p^{\zeta-1}} \l\| u^+\r\|_{1,p}^{p\zeta}-\int_{\Omega_+}  \frac{f\l(x,\alpha u^+\r)}{\l(\alpha u^+\r)^{p\zeta-1}}\l(u^+\r)^{p\zeta}\,\mathrm{d} x\\
		&\quad-\int_{(\partial \Omega)_+}  \frac{g\l(x,\alpha u^+\r)}{\l(\alpha u^+\r)^{p\zeta-1}}\l(u^+\r)^{p\zeta}\,\mathrm{d} \sigma,
	\end{align*}
	where the second and the third terms on the right-hand side converge to zero as $\alpha\to 0$ because of assumption \eqref{H3}\eqref{H3iii}. Thus,
	\begin{align*}
		\frac{\partial \Lambda_u(\alpha,\beta_0)}{\partial \alpha}>0  \quad\text{for }\alpha>0 \text{ sufficiently small}.
	\end{align*}
	Hence, $\Lambda_u$ is increasing with respect to $\alpha$ which is a contradiction. Analogously, we can prove that $\Lambda_u$ cannot achieve its global maximum at a point $(\alpha_0,0)$ with $\alpha_0\geq 0$. In summary, we have seen that the global maximum must be achieved in $(0,L)^2$ for some $L>0$. But this means it is a critical point of $\Lambda_u$ and taking Parts 1 and 3 into account, we know that the only critical point of $\Lambda_u$ is $(\alpha_u,\beta_u)$. This finishes the proof.
\end{proof}

\begin{proposition}
	\label{proposition_nodal_unique-pair-less-one}
	Let hypotheses \eqref{H1}--\eqref{H3} be satisfied and let $u \in  W^{1,\mathcal{H}}(\Omega)$ with $u^{\pm}\neq 0$ such that
	\begin{align}\label{new-5}
		\langle \mathcal{E}'(u), u^+\rangle \leq 0
		\quad\text{and}\quad
		\langle \mathcal{E}'(u),- u^-\rangle \leq 0.
	\end{align}
	Then the unique pair $(\alpha_u,\beta_u)$ obtained in Proposition \ref{proposition_nodal_unique-pair} satisfies $0<\alpha_u,\beta_u \leq 1$. Alternatively, if
	\begin{align}\label{new-6}
		\langle \mathcal{E}'(u), u^+\rangle \geq 0
		\quad\text{and}\quad
		\langle \mathcal{E}'(u),- u^-\rangle \geq 0,
	\end{align}
	then $\alpha_u,\beta_u \geq 1$.
\end{proposition}

\begin{proof}
	We first assume that $0<\beta_u\leq \alpha_u$. Since $\alpha_u u^+-\beta_uu^-$ belongs to $\mathcal{C}$, it holds by definition that $\l\langle\mathcal{E}'\l(\alpha_u u^+-\beta_u u^-\r),\alpha_u u^+ \r\rangle=0$ which is equivalent to
	\begin{equation}
		\label{G13}
		\begin{aligned}
			&\l(a+b\l(\frac{1}{p}\l\|\alpha_u u^+-\beta_u u^- \r\|_{1,p}^p+\frac{1}{q}\l\|\nabla \l(\alpha_u u^+-\beta_u u^-\r) \r\|_{q,\mu}^q\r)^{\zeta-1}\r) \\
		& \quad\times \l(\l\|\alpha_u u^+ \r\|_{1,p}^p+\l\|\nabla \l(\alpha_u u^+\r) \r\|_{q,\mu}^q\r) \\
		& \quad -\int_\Omega  f\l(x,\alpha_u u^+\r)\alpha_u u^+\,\mathrm{d} x-\int_{\partial\Omega}  g\l(x,\alpha_u u^+\r)\alpha_u u^+\,\mathrm{d} \sigma =0.
		\end{aligned}
	\end{equation}
	From \eqref{new-5} we have $\langle \mathcal{E}'(u),u^+\rangle \leq 0$, that is,
	\begin{equation}
		\label{G14}
		\begin{aligned}
			&\l(a+b\l(\frac{1}{p}\l\|u \r\|_{1,p}^p+\frac{1}{q}\l\|\nabla u \r\|_{q,\mu}^q\r)^{\zeta-1}\r)
			\l(\l\|u^+ \r\|_{1,p}^p+\l\|\nabla u^+\r\|_{q,\mu}^q\r)\\
			&-\int_\Omega  f\l(x,u^+\r)u^+\,\mathrm{d} x-\int_{\partial\Omega}  g\l(x, u^+\r) u^+\,\mathrm{d} \sigma \leq 0.
		\end{aligned}
	\end{equation}
	Supposing by contradiction that $\alpha_u > 1$ and dividing \eqref{G13} by $\alpha_u^{q\zeta}$ gives
	\begin{equation}
		\label{G15}
		\begin{aligned}
			0
			&\leq \l(a\frac{1}{\alpha_u^{q(\zeta-1)}}+b\l(\frac{1}{p}\l\| u\r\|_{1,p}^p+\frac{1}{q}\l\|\nabla u \r\|_{q,\mu}^q\r)^{\zeta-1}\r)\\
			&\quad\times\l(\l\| u^+\r\|_{1,p}^p+\l\|\nabla u^+ \r\|_{q,\mu}^q\r)
			\\
			&\quad -\int_{\Omega_+}  \frac{f\l(x,\alpha_u u^+\r)}{\l(\alpha_u u^+\r)^{q\zeta-1}} \l(u^+\r)^{q\zeta}\,\mathrm{d} x
			-\int_{(\partial\Omega)_+}  \frac{g\l(x,\alpha_u u^+\r)}{\l(\alpha_u u^+\r)^{q\zeta-1}} \l(u^+\r)^{q\zeta}\,\mathrm{d} \sigma.
		\end{aligned}
	\end{equation}
	Now, combining \eqref{G14} and \eqref{G15}, it follows that
	\begin{equation*}
		%\label{TH34}
		\begin{aligned}
			&\int_{\Omega_+}  \l(\frac{f\l(x,\alpha_u u^+\r)}{\l(\alpha_u u^+\r)^{q\zeta-1}} - \frac{f\l(x,u^+\r)}{\l(u^+\r)^{q\zeta-1}} \r) \l(u^+\r)^{q\zeta}\,\mathrm{d} x\\
			&+\int_{(\partial\Omega)_+}  \l(\frac{g\l(x,\alpha_u u^+\r)}{\l(\alpha_u u^+\r)^{q\zeta-1}} - \frac{g\l(x,u^+\r)}{\l(u^+\r)^{q\zeta-1}} \r) \l(u^+\r)^{q\zeta}\,\mathrm{d} \sigma\\
			&\leq a\l(\frac{1}{\alpha_u^{q(\zeta-1)}}-1\r)
			\l(\l\| u^+\r\|_{1,p}^p+\l\|\nabla u^+ \r\|_{q,\mu}^q\r).
		\end{aligned}
	\end{equation*}
	But this is a contradiction due to assumption \eqref{H3}\eqref{H3v} (see also \eqref{G8} and the arguments there). Hence, we have $0<\beta_u\leq \alpha_u\le1$. If $0<\alpha_u\leq \beta_u$, we use $\l\langle\mathcal{E}'\l(\alpha_u u^+-\beta_u u^-\r),-\beta_u u^- \r\rangle=0$ and $\langle \mathcal{E}'(u),- u^-\rangle \leq 0$ and work as above assuming $\beta_u > 1$ to get
	\begin{equation*}
		\begin{aligned}
			&-\int_{\Omega_-}  \l(\frac{f\l(x,-\beta_u u^-\r)}{\l(\beta_u u^-\r)^{q\zeta-1}} - \frac{f\l(x,-u^-\r)}{\l(u^-\r)^{q\zeta-1}} \r) \l(u^-\r)^{q\zeta}\,\mathrm{d} x\\
			&-\int_{(\partial\Omega)_-}  \l(\frac{g\l(x,-\beta_u u^-\r)}{\l(\beta_u u^-\r)^{q\zeta-1}} - \frac{g\l(x,-u^-\r)}{\l(u^-\r)^{q\zeta-1}} \r) \l(u^-\r)^{q\zeta}\,\mathrm{d} \sigma\\
			&\leq a\l(\frac{1}{\beta_u^{(\zeta-1)q}}-1\r)
			\l(\l\| u^-\r\|_{1,p}^p+\l\|\nabla u^- \r\|_{q,\mu}^q\r).
		\end{aligned}
	\end{equation*}
	This is again a contradiction and it holds $0<\alpha_u\leq \beta_u \leq 1$. The case \eqref{new-6} can be done in a similar way, just using the inequalities in the opposite direction.
\end{proof}

The following proposition will be useful for later considerations.

\begin{proposition}
	\label{proposition_auxiliary_result}
	Let hypotheses \eqref{H1}--\eqref{H3} be satisfied and let $u\in W^{1,\mathcal{H}}(\Omega)$ with $\|u\|\leq 1$. Then there exist constants $\hat{C}, \tilde{C}_1,\tilde{C}_2>0$ such that
	\begin{align*}
		\mathcal{E}(u) \geq
		\hat{C} \|u\|^{q\zeta} - \tilde{C}_1 \|u\|^{r_1} - \tilde{C}_2 \|u\|^{r_2}.
	\end{align*}
\end{proposition}

\begin{proof}
	With regard to assumptions \eqref{H3}\eqref{H3i} and \eqref{H3}\eqref{H3iii}, for a given $\varepsilon>0$, there exist constants $\hat{c}_1=\hat{c}_1(\varepsilon)$, $\hat{c}_2=\hat{c}_2(\varepsilon)>0$ such that
	\begin{align}
		F(x,s) &\leq \frac{\varepsilon}{p\zeta}|s|^{p\zeta} +\hat{c}_1|s|^{r_1} \quad \text{for a.a.\,}x\in \Omega,\label{l1}\\
		G(x,s) &\leq \frac{\varepsilon}{p\zeta}|s|^{p\zeta} +\hat{c}_2|s|^{r_2} \quad \text{for a.a.\,}x\in \partial\Omega\label{l2}
	\end{align}
	and for all $s\in\R$.
	Let $u \in  W^{1, \mathcal{H} }_0 ( \Omega )$ with $\|u\|\leq 1$. Then, by applying \eqref{l1}, \eqref{l2} and Proposition \ref{proposition_embeddings}(i), (ii), we obtain
	\begin{align*}
		\mathcal{E}(u)
		&=\Phi\l[\Xi(u)\r]-\int_\Omega  F(x,u)\,\mathrm{d} x-\int_{\partial \Omega}  G(x,u)\,\mathrm{d} \sigma\\
		&\geq \frac{b}{\zeta}\l(\frac{1}{p^\zeta}\| u\|_{1,p}^{p\zeta}+\frac{1}{q^\zeta}\|\nabla u\|_{q,\mu}^{q\zeta}\r)
		-\frac{\varepsilon}{p\zeta} \|u\|_{p\zeta}^{p\zeta} - \hat{c}_1 \|u\|_{r_1}^{r_1} \\
		&\quad -\frac{\varepsilon}{p\zeta} \|u\|_{p\zeta,\partial\Omega}^{p\zeta} - \hat{c}_2 \|u\|_{r_2,\partial\Omega}^{r_2}\\
		&\geq \l(\frac{b}{p^{\zeta}\zeta}-\varepsilon\left(\frac{C^{p\zeta}_\Omega+C^{p\zeta}_{\partial \Omega}}{p\zeta}\r)\r)\| u\|_{1,p}^{p\zeta}+\frac{b}{q^{\zeta}\zeta}\|\nabla u\|_{q,\mu}^{q\zeta} \\
		&\quad -\hat{c}_1 \l(C_\Omega^\mathcal{H}\r)^{r_1}\|u\|^{r_1}-\hat{c}_2 \l(C_{\partial \Omega}^\mathcal{H}\r)^{r_2}\|u\|^{r_2}\\
		& \geq \min \l\{\l(\frac{b}{p^{\zeta}\zeta}-\varepsilon\left(\frac{C^{p\zeta}_\Omega+C^{p\zeta}_{\partial \Omega}}{p\zeta}\r)\r),\frac{b}{q^{\zeta}\zeta}\r\} \frac{1}{2^{\zeta-1}} \left[ \rho(u) \right]^\zeta \\
		&\quad -\hat{c}_1 \l(C_\Omega^\mathcal{H}\r)^{r_1}\|u\|^{r_1}-\hat{c}_2 \l(C_{\partial \Omega}^\mathcal{H}\r)^{r_2}\|u\|^{r_2}
	\end{align*}
	with $C_\Omega$, $C_\Omega^{\mathcal{H}}$, $C_{\partial \Omega}$, $C_{\partial \Omega}^{\mathcal{H}}$ being the embedding constants of the continuous embeddings $W^{1,p}(\Omega)\to L^{p\zeta}(\Omega)$, $ W^{1, \mathcal{H} } ( \Omega ) \to L^{r_1}(\Omega)$, $W^{1,p}(\Omega)\to L^{p\zeta}(\partial\Omega)$ and $ W^{1, \mathcal{H} }( \Omega ) \to L^{r_2}(\partial\Omega)$, respectively, where we have used the inequality $2^{1-\zeta}(s+t)^\zeta \leq s^\zeta + t^\zeta$ for all $s,t \geq 0$ in the last step. Choosing $\varepsilon \in \left(0, \frac{b}{p^{\zeta-1}(C_\Omega^{p\zeta}+C_{\partial \Omega}^{p\zeta})}\right)$ and making use of Proposition \ref{proposition_modular_properties2}(iii) yields the required statement.
\end{proof}

Let $\omega=\inf\limits_{\mathcal{C}}  \mathcal{E}$.

\begin{proposition}\label{proposition_positive_infimum_1}
	Let hypotheses \eqref{H1}--\eqref{H3} be satisfied. Then $\omega>0$. In particular, the infimum is finite.
\end{proposition}

\begin{proof}
	Recall that $q\zeta < \min\{r_1,r_2\}$. Then from Proposition \ref{proposition_auxiliary_result} we know that
	\begin{align}
		\label{new-5465415}
		\mathcal{E}(u) \geq C>0 \quad\text{for all }u \in  W^{1,\mathcal{H}}(\Omega) \text{ with }\|u\|=\tau
	\end{align}
	for some $\tau\in (0,1)$ sufficiently small. Let $u \in \mathcal{C}$ and choose $\alpha_0, \beta_0>0$ such that $\|\alpha_0 u^+-\beta_0u^-\|=\tau$. Then, using \eqref{new-5465415} and Proposition \ref{proposition_nodal_unique-pair}, it follows that
	\begin{align*}
		0<C \leq \mathcal{E}(\alpha_0 u^+-\beta_0u^-) \leq \mathcal{E}(u).
	\end{align*}
	As $u \in \mathcal{C}$ was arbitrary, we conclude that $\omega>0$.
\end{proof}

In the next proposition we show that $\mathcal{E}$ restricted to the constraint set $\mathcal{C}$ is sequentially coercive, that is, $\mathcal{E}(u_n) \to \infty$ for any sequence $\{u_n\}_{n\in\N}\subseteq \mathcal{C}$ with $\|u_n\|\to+\infty$.

\begin{proposition}
	\label{proposition_nodal_positive-infimum}
	Let hypotheses \eqref{H1}--\eqref{H3} be satisfied. Then $\mathcal{E}|_{\mathcal{C}}$ is sequentially coercive.
\end{proposition}

\begin{proof}
	Let $\{u_n\}_{n\in\N}\subseteq \mathcal{C}$ be a sequence satisfying
	\begin{align}\label{new-11}
		\|u_n\|\to+\infty \quad\text{as }n\to\infty
	\end{align}
	and let $y_n=\frac{u_n}{\|u_n\|}$. Then we have
	\begin{equation}
		\label{G16}
		\begin{aligned}
			 & y_n \rightharpoonup y \quad\text{in } W^{1,\mathcal{H}}(\Omega),\\
			 &y_n\to y \quad\text{in }L^{r_1}(\Omega) \text{ and a.e.\,in }\Omega, \\
			 &y_n\to y \quad\text{in }L^{r_2}(\partial\Omega) \text{ and a.e.\,in }\partial\Omega, \\
			 & y_n^{\pm} \rightharpoonup  y^{\pm} \quad\text{in } W^{1,\mathcal{H}}(\Omega),\\
			 &y_n^{\pm}\to y^{\pm} \quad\text{in }L^{r_1}(\Omega)\text{ and a.e.\,in }\Omega,\\
			 &y_n^{\pm}\to y^{\pm} \quad\text{in }L^{r_2}(\partial\Omega)\text{ and a.e.\,in }\partial\Omega,
		\end{aligned}
	\end{equation}
	for some $y=y^+-y^-\in  W^{1,\mathcal{H}}(\Omega)$. Let us suppose first that $y\neq 0$. Applying Proposition \ref{proposition_modular_properties2}(iv) and supposing $\|u_n\| > 1$ yields
	\begin{equation}
		\label{n14}
		\begin{aligned}
			\mathcal{E}(u_n)
			&=a\l(\frac{1}{p}\l\| u_n\r\|_{1,p}^p+\frac{1}{q}\l\|\nabla u_n\r\|_{q,\mu}^q\r)\\
			&\quad +\frac{b}{\zeta}\l(\frac{1}{p}\l\| u_n\r\|_{1,p}^p+\frac{1}{q}\l\|\nabla u_n\r\|_{q,\mu}^q\r)^\zeta\\
			& \quad -\int_\Omega  F(x,u_n)\,\mathrm{d} x-\int_{\partial \Omega}  G(x,u_n)\,\mathrm{d} \sigma\\
			&\leq \frac{a}{p}\|u_n\|^q+\frac{b}{\zeta p^\zeta}\|u_n\|^{q\zeta} -\int_\Omega  F(x,u_n)\,\mathrm{d} x-\int_{\partial \Omega}  G(x,u_n)\,\mathrm{d} \sigma.
		\end{aligned}
	\end{equation}

	Let
	\begin{align*}
		&\Omega^y_+=\left\{x\in \Omega \colon y(x)>0 \right\}
		\quad\text{and}\quad
		\Omega^y_-=\left\{x\in \Omega \colon y(x)<0 \right\}\\
		&\Gamma^y_+=\left\{x\in \partial\Omega \colon y(x)>0 \right\}
		\quad\text{and}\quad
		\Gamma^y_-=\left\{x\in \partial\Omega \colon y(x)<0 \right\}.
	\end{align*}
	At least one of $\Omega^y_+$ and $\Omega^y_-$ has a positive measure. Suppose $|\Omega^y_+|_N>0$. Then we have, because of \eqref{G16}, that
	\begin{align*}
		u_n^+(x) \to +\infty \quad \text{for a.a.\,}x\in\Omega^y_+.
	\end{align*}
	Using Fatou's Lemma, \eqref{new-11} and   \eqref{superlinear-F} we obtain
	\begin{align}
		\label{c10}
		\int_{\Omega^y_+} \frac{F(x,u_n^+)}{\l\|u_n^+\r\|^{q\vartheta}}\,\mathrm{d} x\to +\infty.
	\end{align}
	Taking the assumptions \eqref{H3}\eqref{H3i} and \eqref{H3}\eqref{H3ii} into account, we can find a positive constant $\hat{C}$ such that
	\begin{align}
		\label{new-12}
		F(x,s) \geq -\hat{C}  \quad \text{for a.a.\,}x\in\Omega \text{ and for all }s\in\R.
	\end{align}
	Using \eqref{new-12} it follows that
	\begin{align*}
		\int_{\Omega} \frac{F\l(x,u_n^+\r)}{\l\|u_n^+\r\|^{q\vartheta}}\,\mathrm{d} x
		& = \int_{\Omega^y_+} \frac{F\l(x,u_n^+\r)}{\l\|u_n^+\r\|^{q\vartheta}}\,\mathrm{d} x+\int_{\Omega\setminus \Omega^y_+} \frac{F\l(x,u_n^+\r)}{\l\|u_n^+\r\|^{q\vartheta}}\,\mathrm{d} x \\
		& \geq \int_{\Omega^y_+} \frac{F\l(x,u_n^+\r)}{\l\|u_n^+\r\|^{q\vartheta}}\,\mathrm{d} x - \frac{\hat{C}}{\l\|u_n^+\r\|^{q\vartheta}}|\Omega|_N.
	\end{align*}
	Combining this with \eqref{new-11} and \eqref{c10} leads to
	\begin{align}
		\label{c15}
		\int_{\Omega} \frac{F\l(x,u_n^+\r)}{\l\|u_n^+\r\|^{q\vartheta}}\,\mathrm{d} x\to +\infty.
	\end{align}
	If the measures of $\Omega^y_-$ and of $\Gamma^y_+, \Gamma^y_-$ are positive, we get the same as in \eqref{c10} for these measures while for $\Gamma^y_+, \Gamma^y_-$ we take the  Hausdorff surface measure. If one of these sets does not have a positive measure, then the corresponding limit is nonnegative as one can see from the treatment above. Altogether, we have
	\begin{equation}\label{new-13}
		\begin{aligned}
			\int_{\Omega} \frac{F\l(x,-u_n^-\r)}{\l\|u_n^-\r\|^{q\vartheta}}\,\mathrm{d} x\to +\infty
			\quad&\text{or}\quad
			\int_{\Omega} \frac{F\l(x,-u_n^-\r)}{\l\|u_n^-\r\|^{q\vartheta}}\,\mathrm{d} x\to C_1\geq 0\\
			\int_{\partial\Omega} \frac{G\l(x,u_n^+\r)}{\l\|u_n^+\r\|^{q\vartheta}}\,\mathrm{d} \sigma\to +\infty
			\quad&\text{or}\quad
			\int_{\partial\Omega} \frac{G\l(x,u_n^+\r)}{\l\|u_n^+\r\|^{q\vartheta}}\,\mathrm{d} \sigma\to C_2\geq 0\\
			\int_{\partial\Omega} \frac{G\l(x,-u_n^-\r)}{\l\|u_n^-\r\|^{q\vartheta}}\,\mathrm{d} \sigma\to +\infty
			\quad&\text{or}\quad
			\int_{\partial\Omega} \frac{G\l(x,-u_n^-\r)}{\l\|u_n^-\r\|^{q\vartheta}}\,\mathrm{d} \sigma\to C_3\geq 0.
	\end{aligned}
	\end{equation}
	Then, dividing \eqref{n14} by $\|u_n\|^{q\zeta}$, passing to the limit as $n\to \infty$ and applying \eqref{c15} as well as \eqref{new-13} it follows that $\frac{\mathcal{E}(u_n)}{\|u_n\|^{q\zeta}}\to -\infty$, a contradiction to $\mathcal{E}(u_n)\geq \omega>0$ for all $n \in \N$, see Proposition \ref{proposition_positive_infimum_1}. Therefore, $y=0$ and so $y^+=y^-=0$. Recall again that $u_n \in \mathcal{C}$. Using this fact along with Proposition \ref{proposition_nodal_unique-pair} and \eqref{G16}, we have for every pair $(t_1,t_2)\in (0,\infty)\times(0,\infty)$ with $0<t_1\leq t_2$ that
	\begin{align*}
		\mathcal{E}(u_n)
		 & \geq \mathcal{E}(t_1 y_n^+-t_2y_n^-)  \\
		 & =a\l(\frac{1}{p}\l\|t_1y_n^+-t_2y_n^-\r\|_{1,p}^p+\frac{1}{q}\l\|\nabla \l(t_1y_n^+-t_2y_n^-\r)\r\|_{q,\mu}^q\r) \\
		 & \quad +\frac{b}{\zeta}\l(\frac{1}{p}\l\|t_1y_n^+-t_2y_n^-\r\|_{1,p}^p+\frac{1}{q}\l\|\nabla \l(t_1y_n^+-t_2y_n^-\r)\r\|_{q,\mu}^q\r)^\zeta \\
		 & \quad -\int_\Omega  F\l(x,t_1y_n^+-t_2y_n^-\r)\,\mathrm{d} x
		 -\int_{\partial \Omega}  G\l(x,t_1y_n^+-t_2y_n^-\r)\,\mathrm{d} \sigma \\
		 & \geq \frac{b}{q^\zeta \zeta}\min \l\{t_1^{p\zeta},t_1^{q\zeta}\r\} \left[ \rho\l(y_n\r) \right] ^\zeta
		-\int_\Omega  F\l(x,t_1y_n^+\r)\,\mathrm{d} x
		-\int_\Omega  F\l(x,-t_2y_n^-\r)\,\mathrm{d} x\\
		&\quad -\int_{\partial \Omega}  G\l(x,t_1y_n^+\r)\,\mathrm{d} \sigma -\int_{\partial \Omega}  G\l(x,-t_2y_n^-\r)\,\mathrm{d} \sigma\to \frac{b}{q^\zeta \zeta}\min \l\{t_1^{p\zeta},t_1^{q\zeta}\r\},
	\end{align*}
	where we have used Proposition \ref{proposition_modular_properties2}(ii) in the last step since $\|y_n\|=1$ which is equivalent to $\rho\l(y_n\r)=1$. Hence, for any given $K>0$ we take $t_1>0$ large enough and then for $n \geq n_0=n_0(t_1)$ we have that $\mathcal{E}(u_n)>K$.
\end{proof}

\begin{proposition}
	\label{proposition_bound_elements_nehari}
	Let hypotheses \eqref{H1}--\eqref{H3} be satisfied. Then there exists $\Psi>0$ such that $\|u^{\pm}\|\geq \Psi>0$ for all $u\in\mathcal{C}$.
\end{proposition}

\begin{proof}
	Let $u\in \mathcal{C}$ such that $\|u^{\pm}\| <1$.  From the definition of $\mathcal{C}$ we know that
	\begin{equation}
		\label{prop_bound_1}
		\begin{aligned}
			 & \left(a+b\left(\frac{1}{p}\left\|u^+-u^-\right\|_{1,p}^p+\frac{1}{q}\left\|\nabla \left(u^+-u^-\right)\right\|_{q,\mu}^q\right)^{\zeta-1}\right) \\
			 & \quad\times\left(\left\|u^+ \right\|_{1,p}^p+\left\|\nabla u^+ \right\|_{q,\mu}^q\right)\\
			 & = \int_\Omega f(x,u^+)u^+ \,\mathrm{d} x+\int_{\partial \Omega} g(x,u^+)u^+ \,\mathrm{d} \sigma.
		\end{aligned}
	\end{equation}
	With regard to \eqref{H3}\eqref{H3i} and \eqref{H3}\eqref{H3iii}, for a given $\varepsilon>0$, there exist $C_{1,\varepsilon}, C_{2,\varepsilon}>0$ such that
	\begin{equation}\label{prop_bound_3}
		\begin{aligned}
			|f(x,s)| &\leq \varepsilon |s|^{p\zeta-1} +C_{1,\varepsilon}|s|^{r_1-1}\quad\text{for a.a.\,}x\in\Omega, \\
			|g(x,s)| &\leq \varepsilon |s|^{p\zeta-1} +C_{2,\varepsilon}|s|^{r_2-1}\quad\text{for a.a.\,}x\in\partial\Omega
		\end{aligned}
	\end{equation}
	and for all $s\in \R$. Applying \eqref{prop_bound_3} in \eqref{prop_bound_1} together  with the continuous embeddings $W^{1,p}(\Omega)\to L^{p\zeta}(\Omega)$, $ W^{1, \mathcal{H} } ( \Omega ) \to L^{r_1}(\Omega)$, $W^{1,p}(\Omega)\to L^{p\zeta}(\partial\Omega)$ and $ W^{1, \mathcal{H} }( \Omega ) \to L^{r_2}(\partial\Omega)$ and related embedding constants $C_\Omega$, $C_\Omega^{\mathcal{H}}$, $C_{\partial \Omega}$, $C_{\partial \Omega}^{\mathcal{H}}$, respectively, we get
	\begin{align*}
		&\frac{b}{q^{\zeta-1}}\left(\left\|u^+ \right\|_{1,p}^{p\zeta}+\left\|\nabla u^+ \right\|_{q,\mu}^{q \zeta}\right)\\
		&\leq \frac{b}{q^{\zeta-1}}\left(\left\|u^+ \right\|_{1,p}^{p}+\left\|\nabla u^+ \right\|_{q,\mu}^{q}\right)^\zeta\\
		&\leq \varepsilon \left\|u^+\right\|_{p\zeta}^{p\zeta}+C_{1,\varepsilon} \left\|u^+\right\|_{r_1}^{r_1}+\varepsilon \left\|u^+\right\|_{p\zeta,\partial\Omega}^{p\zeta}+C_{2,\varepsilon} \left\|u^+\right\|_{r_1,\partial\Omega}^{r_1}\\
		&\leq \varepsilon \l(C_\Omega^{p\zeta}+C_{\partial \Omega}^{p\zeta}\r) \left\| u^+\right\|_{1,p}^{p\zeta}+C_{1,\varepsilon} \l(C_\Omega^{\mathcal{H}}\r)^{r_1}\left\|u^+\right\|^{r_1}+C_{2,\varepsilon} \l(C_{\partial \Omega}^{\mathcal{H}}\r)^{r_2}\left\|u^+\right\|^{r_2}.
	\end{align*}
	Now we choose $\varepsilon \in \left( 0,\frac{b}{q^{\zeta-1}\l(C_\Omega^{p\zeta}+C_{\partial \Omega}^{p\zeta}\r)}\right)$ and apply the inequality $2^{1-\zeta}(s+t)^\zeta \leq s^\zeta + t^\zeta$ for all $s,t \geq 0$ as well as Proposition \ref{proposition_modular_properties2}(iii) which results in
	\begin{align*}
		\Psi \left\|u^+\right\|^{q\zeta}
		\leq \Psi \left[ \rho(u^+) \right] ^\zeta
		\leq C_1\left\|u^+\right\|^{r_1}+C_2\left\|u^+\right\|^{r_2}
	\end{align*}
	for some $\Psi, C_1, C_2>0$. Since $q\zeta<\min\{r_1,r_2\}$, see Remark \ref{remark-H2}, the result follows for $u^+$. A similar treatment can be done for the  assertion for $u^-$ which finishes the proof of the proposition.
\end{proof}

Next we are able to prove that the infimum of $\mathcal{E}$ restricted to the constraint set $\mathcal{C}$ is achieved.

\begin{proposition}
	\label{proposition_infimum_achieved}
	Let hypotheses \eqref{H1}--\eqref{H3} be satisfied. Then there exists $y_0\in\mathcal{C}$ such that $\mathcal{E}(y_0)=\omega$.
\end{proposition}

\begin{proof}
	Let $\{y_n\}_{n\in\N}\subseteq \mathcal{C}$ be a minimizing sequence of $\mathcal{E}$ over $\mathcal{C}$, that is,
	\begin{align*}
		\mathcal{E}(y_n) \searrow \omega.
	\end{align*}
	From Proposition \ref{proposition_nodal_positive-infimum} it is clear that the sequence $\{y_n\}_{n\in\N}$ is bounded in $W^{1,\mathcal{H}}(\Omega)$. This implies that $\{y_n^+\}_{n\in\N}, \{y_n^-\}_{n\in\N}$ are bounded in $W^{1,\mathcal{H}}(\Omega)$ as well (see Proposition \ref{proposition_modular_properties2} \textnormal{(vi)}). Then we can assume, for not relabeled subsequences, that
	\begin{equation}
		\label{prop_inf_achieved_1}
		\begin{aligned}
			y_n^+ &\rightharpoonup y_0^+ \quad\text{in }W^{1,\mathcal{H}}(\Omega),  \quad y_0^+ \geq 0, \\
			y_n^- &\rightharpoonup y_0^- \quad\text{in }W^{1,\mathcal{H}}(\Omega),  \quad y_0^- \geq 0, \\
			y_n^+&\to y_0^+ \quad\text{in }L^{r_1}(\Omega), L^{r_2}(\partial\Omega), \text{ a.e.\,in }\Omega \text{ and a.e.\,in } \partial\Omega, \\
			y_n^-&\to y_0^- \quad\text{in }L^{r_1}(\Omega), L^{r_2}(\partial\Omega), \text{ a.e.\,in }\Omega \text{ and a.e.\,in } \partial\Omega.
		\end{aligned}
	\end{equation}
	From \eqref{prop_inf_achieved_1} and assumption \eqref{H3}\eqref{H3i} we know that, for all $\alpha, \beta >0$,
	\begin{equation}
		\label{prop_inf_achieved_2}
		\begin{aligned}
			\int_\Omega f(x,\alpha y_n^+)\alpha y_n^+\,\mathrm{d} x
			 & \to \int_\Omega f(x,\alpha y_0^+)\alpha y_0^+\,\mathrm{d}x, \\
			 \int_{\partial \Omega} g(x,\alpha y_n^+)\alpha y_n^+\,\mathrm{d} \sigma
			 & \to \int_{\partial \Omega} g(x,\alpha y_0^+)\alpha y_0^+\,\mathrm{d}\sigma, \\
			\int_\Omega f(x,-\beta y_n^-)(-\beta y_n^-)\,\mathrm{d} x
			 & \to \int_\Omega f(x,-\beta y_0^-)(-\beta y_0^-)\,\mathrm{d}x,\\
			\int_{\partial \Omega} g(x,-\beta y_n^-)(-\beta y_n^-)\,\mathrm{d} \sigma
			& \to \int_{\partial \Omega} g(x,-\beta y_0^-)(-\beta y_0^-)\,\mathrm{d}\sigma
		\end{aligned}
	\end{equation}
	as $n\to \infty$ and also
	\begin{equation}
		\label{prop_inf_achieved_3}
		\begin{aligned}
			\int_\Omega F(x,\alpha y_n^+)\,\mathrm{d} x
			 & \to \int_\Omega F(x,\alpha y_0^+)\,\mathrm{d}x,  \\
			 \int_{\partial \Omega} G(x,\alpha y_n^+)\,\mathrm{d} \sigma
			 & \to \int_{\partial \Omega} G(x,\alpha y_0^+)\,\mathrm{d}\sigma,  \\
			\int_\Omega F(x,-\beta y_n^-)\,\mathrm{d} x
			 & \to \int_\Omega F(x,-\beta y_0^-)\,\mathrm{d}x,\\
			 \int_{\partial \Omega} G(x,-\beta y_n^-)\,\mathrm{d} \sigma
			 & \to \int_{\partial \Omega} G(x,-\beta y_0^-)\,\mathrm{d}\sigma
		\end{aligned}
	\end{equation}
	as $n\to \infty$.

	{\bf Claim:} $y_0^+\neq 0 \neq y_0^-$

	We prove the Claim via contradiction and assume that $y_0^+=0$. By assumption, $y_n\in \mathcal{C}$, hence
	\begin{align*}
		0
		& =\l\langle \mathcal{E}'\l(y_n\r), y_n^+\r\rangle \\
		& =\l(a+b\l(\frac{1}{p}\l\| y_n \r\|_{1,p}^p+\frac{1}{q}\l\|\nabla y_n \r\|_{q,\mu}^q\r)^{\zeta-1}\r) \\
		& \quad \times \l(\l\| y_n^+\r\|_{1,p}^p+\l\|\nabla y_n^+ \r\|_{q,\mu}^q\r) -\int_\Omega  f\l(x,y_n^+\r)y_n^+\,\mathrm{d} x  -\int_{\partial \Omega}  g\l(x,y_n^+\r)y_n^+\,\mathrm{d} \sigma\\
		& \geq \frac{b}{q^{\zeta-1}}\l(\l\| y_n^+\r\|_{1,p}^p+\l\|\nabla y_n^+ \r\|_{q,\mu}^q\r)^\zeta-\int_\Omega  f\l(x, y_n^+\r) y_n^+\,\mathrm{d} x-\int_{\partial \Omega}  g\l(x,y_n^+\r)y_n^+\,\mathrm{d} \sigma.
	\end{align*}
	Using the convergence properties in  \eqref{prop_inf_achieved_2} we conclude from the inequality above that
	\begin{align*}
		\frac{b}{q^{\zeta-1}} \left[ \rho(y_n^+) \right] ^\zeta
		\leq \int_\Omega  f\l(x, y_n^+\r) y_n^+\,\mathrm{d} x +\int_{\partial \Omega}  g\l(x,y_n^+\r)y_n^+\,\mathrm{d} \sigma\to 0
	\end{align*}
	as $n\to \infty$. Thus, $\rho(y_n^+)\to 0$ as $n\to\infty$ which is equivalent to  $y_n^+\to 0$ in $W^{1,\mathcal{H}}(\Omega)$, see Proposition \ref{proposition_modular_properties2}(v). On the other hand, we know from Proposition \ref{proposition_bound_elements_nehari} that $\|y_n^+\|\geq \Psi>0$ which is a contradiction. Similarly, we can prove that $y_0^-\neq 0$. The Claim is proved.

	From the Claim and Proposition \ref{proposition_nodal_unique-pair} there exist unique $\alpha_{y_0}, \beta_{y_0}>0$ such that $\alpha_{y_0}y_0^+ -\beta_{y_0}y_0^-\in \mathcal{C}$. Moreover, using the convergence properties in \eqref{prop_inf_achieved_2} along with the weak lower semicontinuity of the norm $\|\cdot\|_{1,p}$ and the seminorm $\|\cdot\|_{q,\mu}$, we obtain
	\begin{align*}
		& \left\langle \mathcal{E}'(y_0),\pm y_0^{\pm} \right\rangle  \\
		& = \l(a+b\l(\frac{1}{p}\l\|y_0 \r\|_{1,p}^p+\frac{1}{q}\l\|\nabla y_0 \r\|_{q,\mu}^q\r)^{\zeta-1}\r) \l(\l\|\pm y_0^{\pm} \r\|_{1,p}^p+\l\|\nabla \left(\pm y_0^{\pm}\right) \r\|_{q,\mu}^q\r) \\
		& \quad -\int_\Omega  f\l(x, \pm y_0^{\pm} \r) \left(\pm y_0^{\pm}\right)\,\mathrm{d} x-\int_{\partial \Omega}  g\l(x, \pm y_0^{\pm} \r) \left(\pm y_0^{\pm}\right)\,\mathrm{d} \sigma  \\
		& \leq \liminf_{n\to \infty}\l(a+b\l(\frac{1}{p}\l\|y_n \r\|_{1,p}^p+\frac{1}{q}\l\|\nabla y_n \r\|_{q,\mu}^q\r)^{\zeta-1}\r)  \\
		& \qquad\qquad\quad\times\l(\l\|\pm y_n^{\pm} \r\|_{1,p}^p+\l\|\nabla \left(\pm y_n^{\pm}\right) \r\|_{q,\mu}^q\r) \\
		& \quad -\lim_{n\to \infty}\int_\Omega  f\l(x, \pm y_n^{\pm} \r) \left(\pm y_n^{\pm}\right)\,\mathrm{d} x-\lim_{n\to \infty}\int_{\partial \Omega}  g\l(x, \pm y_n^{\pm} \r) \left(\pm y_n^{\pm}\right)\,\mathrm{d} \sigma \\
		& =\liminf_{n\to \infty}\left\langle \mathcal{E}'(y_n),\pm y_n^{\pm} \right\rangle=0,
	\end{align*}
	since $y_n \in \mathcal{C}$. This allows us the usage of Proposition \ref{proposition_nodal_unique-pair-less-one} which shows that $ \alpha_{y_0}, \beta_{y_0} \in (0,1]$. From this fact and assumption \eqref{H3}\eqref{H3iv} we conclude that
	\begin{equation}
		\label{prop_inf_achieved_4}
		\begin{aligned}
			 &\frac{1}{q\zeta}f(x,\alpha_{y_0} y_0^+)\alpha_{y_0} y_0^+-F(x,\alpha_{y_0} y_0^+) \leq \frac{1}{q\zeta}f(x,y_0^+) y_0^+-F(x, y_0^+),   \\
			&\frac{1}{q\zeta}f(x,-\beta_{y_0} y_0^-)(-\beta_{y_0} y_0^-)-F(x,-\beta_{y_0} y_0^-)\\
			& \leq \frac{1}{q\zeta}f(x,-y_0^-) (-y_0^-)-F(x, -y_0^-)
		\end{aligned}
	\end{equation}
	for a.a.\,$x\in\Omega$ and
	\begin{equation}
		\label{prop_inf_achieved_4-boundary}
		\begin{aligned}
			&\frac{1}{q\zeta}g(x,\alpha_{y_0} y_0^+)\alpha_{y_0} y_0^+-G(x,\alpha_{y_0} y_0^+) \leq \frac{1}{q\zeta}g(x,y_0^+) y_0^+-G(x, y_0^+),   \\
			&\frac{1}{q\zeta}g(x,-\beta_{y_0} y_0^-)(-\beta_{y_0} y_0^-)-G(x,-\beta_{y_0} y_0^-)\\
			& \leq \frac{1}{q\zeta}g(x,-y_0^-) (-y_0^-)-G(x, -y_0^-)
		\end{aligned}
	\end{equation}
	for a.a.\,$x\in\partial\Omega$. Now, from $\alpha_{y_0}y_0^+ -\beta_{y_0}y_0^-\in \mathcal{C}$, $ \alpha_{y_0}, \beta_{y_0} \in (0,1]$, \eqref{prop_inf_achieved_2}, \eqref{prop_inf_achieved_3}, \eqref{prop_inf_achieved_4}, \eqref{prop_inf_achieved_4-boundary} and $y_n\in\mathcal{C}$ we derive that
	\begin{align*}
		\omega
		 & \leq \mathcal{E}\left(\alpha_{y_0}y_0^+-\beta_{y_0}y_0^-\right)
		-\frac{1}{q\zeta} \left\langle \mathcal{E}'\left(\alpha_{y_0}y_0^+-\beta_{y_0}y_0^-\right),\alpha_{y_0}y_0^+-\beta_{y_0}y_0^- \right\rangle   \\
		 & =a\l(\frac{\alpha_{y_0}^p}{p}\l\| y_0^+\r\|_{1,p}^p+\frac{\beta_{y_0}^p}{p}\l\| y_0^-\r\|_{1,p}^p+\frac{\alpha_{y_0}^q}{q}\l\|\nabla y_0^+\r\|_{q,\mu}^q+\frac{\beta_{y_0}^q}{q}\l\|\nabla y_0^-\r\|_{q,\mu}^q\r)  \\
		 & \quad +\frac{b}{\zeta}\l(\frac{\alpha_{y_0}^p}{p}\l\| y_0^+\r\|_{1,p}^p+\frac{\beta_{y_0}^p}{p}\l\| y_0^-\r\|_{1,p}^p+\frac{\alpha_{y_0}^q}{q}\l\|\nabla y_0^+\r\|_{q,\mu}^q+\frac{\beta_{y_0}^q}{q}\l\|\nabla y_0^-\r\|_{q,\mu}^q\r)^{\zeta} \\
		 & \quad -\int_\Omega  F\l(x,\alpha_{y_0} u^+\r)\,\mathrm{d} x-\int_\Omega  F\l(x,-\beta_{y_0} u^-\r)\,\mathrm{d} x  \\
		 & \quad -\int_{\partial \Omega}  G\l(x,\alpha_{y_0} u^+\r)\,\mathrm{d} \sigma-\int_{\partial\Omega}  G\l(x,-\beta_{y_0} u^-\r)\,\mathrm{d} \sigma  \\
		 & \quad -a\left(\frac{\alpha_{y_0}^p}{q\zeta}\l\| y_0^+\r\|_{1,p}^p+\frac{\beta_{y_0}^p}{q\zeta}\l\| y_0^-\r\|_{1,p}^p+\frac{\alpha_{y_0}^q}{q\zeta}\l\|\nabla y_0^+\r\|_{q,\mu}^q+\frac{\beta_{y_0}^q}{q\zeta}\l\|\nabla y_0^-\r\|_{q,\mu}^q\right)   \\
		 & \quad -\frac{b}{q\zeta}\l(\frac{\alpha_{y_0}^p}{p}\l\| y_0^+\r\|_{1,p}^p+\frac{\beta_{y_0}^p}{ p}\l\| y_0^-\r\|_{1,p}^p+\frac{\alpha_{y_0}^q}{q}\l\|\nabla y_0^+\r\|_{q,\mu}^q+\frac{\beta_{y_0}^q}{ q}\l\|\nabla y_0^-\r\|_{q,\mu}^q\r)^{\zeta-1} \\
		 & \qquad\times\left(\alpha_{y_0}^p\l\| y_0^+\r\|_{1,p}^p+\beta_{y_0}^p\l\| y_0^-\r\|_{1,p}^p+\alpha_{y_0}^q\l\|\nabla y_0^+\r\|_{q,\mu}^q+\beta_{y_0}^q\l\|\nabla y_0^-\r\|_{q,\mu}^q\right)   \\
		 & \quad+\frac{1}{q\zeta}\int_\Omega f\left(x,\alpha_{y_0} y_0^+\right)\alpha_{y_0} y_0^+\,\mathrm{d} x
		+\frac{1}{q\zeta}\int_\Omega f\left(x,-\beta_{y_0} y_0^-\right)(-\beta_{y_0} y_0^-)\,\mathrm{d} x  \\
		& \quad+\frac{1}{q\zeta}\int_{\partial \Omega} g\left(x,\alpha_{y_0} y_0^+\right)\alpha_{y_0} y_0^+\,\mathrm{d} \sigma
		+\frac{1}{q\zeta}\int_{\partial \Omega} g\left(x,-\beta_{y_0} y_0^-\right)(-\beta_{y_0} y_0^-)\,\mathrm{d} \sigma \\
		 & =a \left(\frac{1}{p}-\frac{1}{q\zeta}\right) \left(\alpha_{y_0}^p\l\| y_0^+\r\|_{1,p}^p+\beta_{y_0}^p\l\| y_0^-\r\|_{1,p}^p \right) \\
		 & \quad +a \left(\frac{1}{q}-\frac{1}{q\zeta}\right) \left(\alpha_{y_0}^q\l\|\nabla y_0^+\r\|_{q,\mu}^q+\beta_{y_0}^q\l\|\nabla y_0^-\r\|_{q,\mu}^q \right) \\
		 & \quad +b\l(\frac{\alpha_{y_0}^p}{p}\l\| y_0^+\r\|_{1,p}^p+\frac{\beta_{y_0}^p}{ p}\l\| y_0^-\r\|_{1,p}^p+\frac{\alpha_{y_0}^q}{q}\l\|\nabla y_0^+\r\|_{q,\mu}^q+\frac{\beta_{y_0}^q}{ q}\l\|\nabla y_0^-\r\|_{q,\mu}^q\r)^{\zeta-1}  \\
		 & \quad \times \Bigg [ \left(\frac{1}{p\zeta}-\frac{1}{q\zeta}\right) \left(\alpha_{y_0}^p\l\| y_0^+\r\|_{1,p}^p+\beta_{y_0}^p\l\|y_0^-\r\|_{1,p}^p \right)  \\
		 & \qquad \quad  +\left(\frac{1}{q\zeta}-\frac{1}{q\zeta}\right) \left(\alpha_{y_0}^q\l\|\nabla y_0^+\r\|_{q,\mu}^q+\beta_{y_0}^q\l\|\nabla y_0^-\r\|_{q,\mu}^q \right)\Bigg]  \\
		 & \quad +\int_\Omega \left(\frac{1}{q\zeta}f(x,\alpha_{y_0} y_0^+)\alpha_{y_0} y_0^+-F(x,\alpha_{y_0} y_0^+)\right)\,\mathrm{d} x  \\
		 & \quad +\int_\Omega \left(\frac{1}{q\zeta}f(x,-\beta_{y_0} y_0^-)(-\beta_{y_0} y_0^-)-F(x,-\beta_{y_0} y_0^-)\right)\,\mathrm{d} x   \\
		 & \quad +\int_{\partial \Omega} \left(\frac{1}{q\zeta}g(x,\alpha_{y_0} y_0^+)\alpha_{y_0} y_0^+-G(x,\alpha_{y_0} y_0^+)\right)\,\mathrm{d} \sigma  \\
		 & \quad +\int_{\partial \Omega} \left(\frac{1}{q\zeta}g(x,-\beta_{y_0} y_0^-)(-\beta_{y_0} y_0^-)-G(x,-\beta_{y_0} y_0^-)\right)\,\mathrm{d} \sigma   \\
		 & \leq a \left(\frac{1}{p}-\frac{1}{q\zeta}\right) \left(\l\| y_0^+\r\|_{1,p}^p+\l\|y_0^-\r\|_{1,p}^p \right)    \\
		 & \quad +a \left(\frac{1}{q}-\frac{1}{q\zeta}\right) \left(\l\|\nabla y_0^+\r\|_{q,\mu}^q+\l\|\nabla y_0^-\r\|_{q,\mu}^q \right)   \\
		 & \quad +b\l(\frac{1}{p}\l\|y_0^+\r\|_{1,p}^p+\frac{1}{ p}\l\|y_0^-\r\|_{1,p}^p+\frac{1}{q}\l\|\nabla y_0^+\r\|_{q,\mu}^q+\frac{1}{ q}\l\|\nabla y_0^-\r\|_{q,\mu}^q\r)^{\zeta-1}  \\
		 & \quad \times \Bigg [ \left(\frac{1}{p\zeta}-\frac{1}{q\zeta}\right) \left(\l\| y_0^+\r\|_{1,p}^p+\l\| y_0^-\r\|_{1,p}^p \right)  \\
		 & \qquad \quad  +\left(\frac{1}{q\zeta}-\frac{1}{q\zeta}\right) \left(\l\|\nabla y_0^+\r\|_{q,\mu}^q+\l\|\nabla y_0^-\r\|_{q,\mu}^q \right)\Bigg]  \\
		 & \quad +\int_\Omega \left(\frac{1}{q\zeta}f(x,y_0^+)y_0^+-F(x,y_0^+)\right)\,\mathrm{d} x   \\
		 & \quad
		+\int_\Omega \left(\frac{1}{q\zeta}f(x,-y_0^-)(-y_0^-)-F(x,- y_0^-)\right)\,\mathrm{d} x   \\
		& \quad +\int_{\partial \Omega} \left(\frac{1}{q\zeta}g(x,y_0^+)y_0^+-G(x,y_0^+)\right)\,\mathrm{d} \sigma   \\
		& \quad
		+\int_{\partial \Omega} \left(\frac{1}{q\zeta}g(x,-y_0^-)(-y_0^-)-G(x,- y_0^-)\right)\,\mathrm{d} \sigma   \\
		 & \leq \liminf_{n\to \infty} \left(\mathcal{E}\left(y_n^+-y_n^-\right)
		-\frac{1}{q\zeta} \left\langle \mathcal{E} '\left(y_n^+-y_n^-\right),y_n^+-y_n^- \right\rangle \right)=\omega,
	\end{align*}
	whereby we have used in the last step the weak lower semicontinuity of the norm $\|\cdot\|_{1,p}$ and the seminorm $\|\cdot\|_{q,\mu}$ along with \eqref{prop_inf_achieved_2} as well as \eqref{prop_inf_achieved_3} and then rearrange the terms inside the limit. Note that if we assume $\alpha_{y_0} < 1$ or $\beta_{y_0} < 1$, the inequality above is strict and this is a contradiction. Hence, we conclude that $\alpha_{y_0}=\beta_{y_0}=1$ which implies that the infimum $\omega$ is achieved by the function $y_0^+-y_0^-$.
\end{proof}

Now we are in the position to give the proof of Theorem \ref{theorem_main_result}.

\begin{proof}[Proof of Theorem \ref{theorem_main_result}]
	Let $y_0\in\mathcal{C}$ with $\mathcal{E}(y_0)=\omega$ be the function obtained in Proposition \ref{proposition_infimum_achieved}. We are going to prove that $y_0$ is a critical point of the functional $\mathcal{E}$ which implies that it is a least energy sign-changing solution of problem \eqref{problem}.

	Arguing by contradiction, we assume that $\mathcal{E} '(y_0)\neq 0$. Then we can find two numbers $\lambda, \delta_0 > 0$ such that
	\begin{align*}
		\|\mathcal{E}'(u)\|_{*} \geq \lambda, \quad
		\text{for all } u \in W^{1,\mathcal{H}}(\Omega) \text{ with } \|u-y_0\| < 3 \delta_0.
	\end{align*}
	Denoting by $C_p$ an embedding constant for $W^{1,\mathcal{H}}(\Omega) \to L^{p}(\Omega)$, we have for any function $v \in  W^{1,\mathcal{H}}(\Omega)$
	\begin{align*}
		\|y_0-v\|\geq C_p^{-1} \|y_0-v\|_p
		\geq
		\begin{cases}
			C_p^{-1} \|y_0^-\|_p, & \text{if } v^- = 0,  \\
			C_p^{-1} \|y_0^+\|_p, & \text{if } v^+ = 0,
		\end{cases}
	\end{align*}
	as $y_0^+ \neq 0 \neq y_0^-$. Next we can choose a number $\delta_1$ such that
	\begin{align*}
		\delta_1 \in \left(0,\min \left\{C_p^{-1} \|y_0^-\|_p,C_p^{-1} \|y_0^+\|_p\right\}\right).
	\end{align*}
	This implies that for any $v \in  W^{1,\mathcal{H}}(\Omega)$ with $\|y_0-v\|< \delta_1$ we have $v^+ \neq 0 \neq v^-$. Now we take $\delta = \min \{ \delta_0, \delta_1 / 2\}$. Because the mapping $(\alpha,\beta) \mapsto \alpha y_0^+ - \beta y_0^-$ is continuous from $[0,\infty)\times [0,\infty)$ into $ W^{1,\mathcal{H}}(\Omega)$, we can find $\tau \in (0,1)$ such that for all $\alpha,\beta \geq 0$ with $\max \{|\alpha - 1|, |\beta - 1| \} < \tau$, it holds
	\begin{align*}
		\left\|\alpha y_0^+ - \beta y_0^- - y_0 \right\| < \delta.
	\end{align*}
	Let $D = ( 1 - \tau, 1 + \tau) \times ( 1 - \tau, 1 + \tau)$ and note that, for any $\alpha,\beta  \geq 0$ with $(\alpha,\beta) \neq  (1,1)$, we get that
	\begin{align}
		\label{Eq:ParametersEstimatedByInfimum}
		\mathcal{E}(\alpha y_0^+ - \beta y_0^-)
		<\mathcal{E}(y_0^+ - y_0^-)=\inf_{u \in \mathcal{C}} \mathcal{E}(u),
	\end{align}
	due to Proposition \ref{proposition_nodal_unique-pair}. This implies that
	\begin{align*}
		\varrho = \max_{(\alpha,\beta) \in \partial D} \mathcal{E}(\alpha y_0^+ - \beta y_0^-)
		<\mathcal{E}(y_0^+ - y_0^-)=\inf_{u \in \mathcal{C}} \mathcal{E}(u).
	\end{align*}

	Next we apply Lemma \ref{Le:DeformationLemma} with the choices
	\begin{align*}
		S = B (y_0, \delta), \quad c = \inf_{u \in \mathcal{C}} \mathcal{E}(u), \quad \varepsilon = \min \left\lbrace \frac{c - \varrho}{4}, \frac{\lambda \delta}{8} \right\rbrace
	\end{align*}
	and $\delta$ is a above. It should be mentioned that $S_{2 \delta} = B (y_0, 3 \delta)$ and by the choice of $\varepsilon$, the hypotheses of Lemma \ref{Le:DeformationLemma} are fulfilled, which guarantees that a mapping $\eta$ with the properties given in the lemma exists. From the choice of $\varepsilon$ it follows that
	\begin{align}
		\label{Eq:2epsAtBoundary}
		\mathcal{E}(\alpha y_0^+ - \beta y_0^-)
		\leq \varrho + c - c
		< c - \left( \frac{c - \varrho}{2} \right)
		\leq c - 2 \varepsilon
	\end{align}
	for all $(\alpha,\beta) \in \partial D$.

	Now we introduce the mappings $\Pi\colon [0,\infty)\times [0,\infty)\to  W^{1,\mathcal{H}}(\Omega)$, $\Upsilon\colon [0,\infty)\times [0,\infty)\to\R^2$ defined by
	\begin{align*}
		\Pi(\alpha,\beta)    & =\eta\l(1,\alpha y_0^+-\beta y_0^-\r)   \\
		\Upsilon(\alpha,\beta ) & =\l(\l\langle \mathcal{E}'\l(\Pi(\alpha,\beta)\r),\Pi^+(\alpha,\beta)\r\rangle,\l\langle \mathcal{E}'\l(\Pi(\alpha,\beta)\r),-\Pi^-(\alpha,\beta)\r\rangle\r).
	\end{align*}
	The continuity of $\eta$ implies the continuity of $\Pi$ and the differentiability of $\mathcal{E}$ guarantees that $\Upsilon$ is continuous. Taking Lemma \ref{Le:DeformationLemma} \textnormal{(i)} and \eqref{Eq:2epsAtBoundary} into account, we know that $\Pi(\alpha,\beta) = \alpha y_0^+ - \beta y_0^-$ for all $(\alpha,\beta) \in \partial D$ and
	\begin{align*}
		\Upsilon(\alpha,\beta) = \left( \langle \mathcal{E}'(\alpha y_0^+-\beta y_0^-) , \alpha y_0^+ \rangle , \langle \mathcal{E}'(\alpha y_0^+-\beta y_0^-) , - \beta y_0^- \rangle \right).
	\end{align*}
	From the sign information in (i)--(iv) of Proposition \ref{proposition_nodal_unique-pair}, we derive the componentwise inequalities, for $\Upsilon=(\Upsilon_1,\Upsilon_2)$,
	\begin{align*}
		 & \Upsilon_1 (1 - \tau,t) > 0 > \Upsilon_1 (1 + \tau,t),  \\
		 & \Upsilon_2 (t, 1 - \tau) > 0 > \Upsilon_2 (t,1 + \tau) \quad
		\text{for all } t \in [1 - \tau, 1 + \tau].
	\end{align*}
	Now we are able to apply Lemma \ref{lemma-poincare-miranda} with $\varphi(\alpha,\beta) = - \Upsilon(1 + \alpha, 1 + \beta)$. This yields a pair $(\alpha_0,\beta_0) \in D$ satisfying $\Upsilon(\alpha_0,\beta_0) = 0$, that is,
	\begin{align*}
		\langle \mathcal{E}'(\Pi(\alpha_0 , \beta_0)) , \Pi^+(\alpha_0 , \beta_0) \rangle = 0 = \langle \mathcal{E}'(\Pi(\alpha_0 , \beta_0)) , -\Pi^- (\alpha_0 , \beta_0) \rangle.
	\end{align*}
	 Lemma \ref{Le:DeformationLemma} \textnormal{(iv)} and the choice of $\tau$ leads to
	\begin{align*}
		\left\|\Pi(\alpha_0,\beta_0)-y_0 \right\| \leq 2 \delta \leq \delta_1
	\end{align*}
	and the choice of $\delta_1$ gives us
	\begin{align*}
		\Pi^+(\alpha_0 , \beta_0) \neq 0 \neq - \Pi^-(\alpha_0 , \beta_0).
	\end{align*}
	This means that $\Pi(\alpha_0,\beta_0) \in \mathcal{C}$. However, by Lemma \ref{Le:DeformationLemma} \textnormal{(ii)}, the choice of $\tau$ and \eqref{Eq:ParametersEstimatedByInfimum}, we have that $\mathcal{E}(\Pi(\alpha_0,\beta_0) ) \leq c - \varepsilon$, which is a contradiction. Therefore, $y_0$ turns out to be a critical point of $\mathcal{E}$. The proof is finished.
\end{proof}

%**********************************************************************************
\section*{Acknowledgment}
%**********************************************************************************

Marcos T.O. Pimenta is partially supported by FAPESP 2023/05300-4, 2023/06617-1 and 2022/16407-1, CNPq 304765/2021-0, Brazil. Marcos T.O. Pimenta thanks the University of Technology Berlin for the kind hospitality during a research stay in February/March 2024. Marcos T.O. Pimenta and Patrick Winkert were financially supported by TU Berlin-FAPESP Mobility Promotion.

\end{document}